\documentclass[12pt, leqno]{amsart}
\usepackage{mycommands}

\title{Verlinde rings and cluster algebras arising from quantum affine algebras}

\author[C.-h. Lee]{Chul-hee Lee}
\address[C.-h. Lee]{June E Huh Center for Mathematical Challenges, Korea Institute for Advanced Study, Seoul 02455, Korea}
\email{chlee@kias.re.kr}

\author[J.-R. Li]{Jian-Rong Li}
\address[J.-R. Li]{Faculty of Mathematics, University of Vienna, Oskar-Morgenstern Platz 1, 1090 Vienna, Austria}
\email{lijr07@gmail.com}

\author[E. Park]{Euiyong Park}
\address[E. Park]{Department of Mathematics, University of Seoul, Seoul 02504, Korea}
\email{epark@uos.ac.kr}

\date{\today}

\begin{document}

\begin{abstract}
We formulate a positivity conjecture relating the Verlinde ring associated with an untwisted affine Lie algebra at a positive integer level and a subcategory of finite-dimensional representations over the corresponding quantum affine algebra with a cluster algebra structure. Specifically, we consider a ring homomorphism from the Grothendieck ring of this representation category to the Verlinde ring and conjecture that every object in the category has a positive image under this map.

We prove this conjecture in certain cases where the underlying simple Lie algebra is simply-laced with level 2 or of type $A_1$ at an arbitrary level. The proof employs the close connection between this category and cluster algebras of finite cluster type. As further evidence for the conjecture, we show that for any level, all objects have positive quantum dimensions under the assumption that some Kirillov-Reshetikhin modules have positive quantum dimensions.
\end{abstract}

\maketitle

\setcounter{tocdepth}{2}
\tableofcontents

\section{Introduction}
The dilogarithm function is defined as  
\[
\dilog(z) = \sum_{n=1}^\infty \frac{z^n}{n^2}, \quad 0 \leq z \leq 1.
\]
Though seemingly simple, it is remarkable that the dilogarithm plays a significant role in various fields, including algebraic \(K\)-theory, number theory, mathematical physics, and hyperbolic geometry.
A \emph{dilogarithm identity} takes the form  
\begin{equation}
\sum_{j\in J} L(x_j) = \frac{\pi^2}{6}c,
\label{eq:dilog}    
\end{equation}
where \(L(z)\) is the Rogers dilogarithm function, given by  
\[
L(z) = \dilog(z) + \frac{1}{2} \log(z) \log(1-z),
\]  
\((x_j)_{j\in J}\) are real algebraic numbers satisfying \(0 \leq x_j \leq 1\), and \(c\) is a rational number;
see, for example, \cite{MR1356515, MR2290758} for a more detailed discussion.

Dilogarithm identities play a crucial role in the study of certain integrable statistical models that are believed to be connected to conformal field theories (CFTs). In CFTs, the central charge of a model is a fundamental parameter that reflects the algebraic structure of the theory. By employing the Thermodynamic Bethe Ansatz (TBA) approach, physicists can link quantities from the statistical models to quantities from the related CFTs; see \cite{MR1092210} for example.
This often leads to the effective central charge \( c \) being computed as a sum of dilogarithms evaluated at certain algebraic numbers, as in \eqref{eq:dilog}. 
While providing a rigorous mathematical justification for the TBA approach is challenging, these dilogarithm identities offer robust consistency checks, thereby supporting the framework's validity.

Building on studies inspired by the TBA framework, Kirillov formulated a family of conjectural dilogarithm identities associated with a simple Lie algebra \(\gzero\) over $\C$, and an integer \(\levk \geq 2\) \cite{MR947332}. The positive real arguments of the dilogarithm identity are given as a solution to a certain system of algebraic equations.
Beyond the dilogarithm identity itself, an interesting aspect of this construction is the fact that these positive arguments could be also derived from the conjecturally positive numbers $\Diim$, obtained by evaluating the characters \(\Qim\) of certain finite-dimensional representations \(\res \Wim\) of \(\gzero\), indexed by \((i, m)\), where \(i\) corresponds to a vertex in the Dynkin diagram of \(\gzero\), and \(m = 0, 1, \ldots, t_i \levk\):  
\[
\Diim := \Qim\left(\frac{\rho}{\levk+\hvee}\right).
\]  
Here, \(\rho\) is the Weyl vector, and \(\hvee\) is the dual Coxeter number.
The value \(t_i \in \{1, 2, 3\}\) is the ratio of the norm of the highest root to the norm of the simple root \(\alpha_i\).
The number obtained by specializing the character of a finite-dimensional \(\gzero\)-module \(V\) as above is also known as its quantum dimension at level \(\levk\), as these values approach the actual dimension of \(V\) as \(\levk \to \infty\).
This quantum dimension is a real number, though it is not necessarily positive.

The development of cluster algebras \cite{MR1887642} led to the eventual proof of the conjectural dilogarithm identity for all types \cite{MR2804544, MR2164838, MR3029994, MR3029995}.
Despite these advancements, expressing the positive arguments in the dilogarithm identity in terms of the character evaluations persisted as an open challenge.
In the meantime, building on Kirillov's work, Kuniba, Nakanishi, and Suzuki explored this problem, and introduced additional conjectures that extended his original formulation \cite{MR1202213, MR1192727, MR1304818, MR2773889}.

In \cite{MR3628223}, the problem was reformulated as one concerning the image of the isomorphism classes of Kirillov-Reshetikhin (KR) modules $\Wim$ over the quantum affine algebra \(\uqghatp\) in the Verlinde ring \(K_0(\fusk)\) associated with the affine Lie algebra $\ghat$ and $\levk$, under a certain ring homomorphism
\[
\phik : K_0(\repuqghat) \to K_0(\fusk).
\]
See \eqref{eq:phik} for its description.
The Verlinde ring \(K_0(\fusk)\) is a \(\Z\)-ring with a distinguished basis, allowing for a well-defined notion of positivity; see \defref{def:verlinde_positivity}.
In particular, an important question was whether \(\phik([\Wim])\) is positive for \(0 \leq m \leq t_i \levk\). 
This formulation has made it possible to prove the conjectures proposed by Kirillov and further developed by Kuniba, Nakanishi, and Suzuki for all classical types.

Although the above formulation has resolved some long-standing questions in certain cases, the new challenge of determining the image of all KR modules in \(K_0(\fusk)\) under the map \(\phik\) remains an open problem; see \conjref{conj:phik} and related references for a more precise and general statement.
Especially, \conjref{conj:phik} introduces an intriguing positivity question concerning \(K_0(\repuqghat)\) and \(K_0(\fusk)\), leading to the following natural questions:
\begin{itemize}
    \item First, is there a natural subring or subset of \(K_0(\repuqghat)\) where the map \(\phik\) consistently produces a positive image? In general, \(\phik([V])\) may be zero or even negative for an object \(V\) in \(\repuqghat\).
    
    \item Second, does this issue have a more direct connection to cluster algebras? Although the original dilogarithm identity is strongly linked to cluster algebras, particularly through the \(Y\)-systems \cite{MR1092210, MR2031858, MR2999039}, it has remained an open problem to place \conjref{conj:phik} within the broader context of cluster algebras.
\end{itemize}

Meanwhile, in their work on the monoidal categorification of cluster algebras, Hernandez and Leclerc introduced certain monoidal subcategories $\Cell$ of $\repuqghat$, parametrized by positive integers $\ell$ \cite{MR2682185, MR3500832}.
They showed that the Grothendieck ring of each of these subcategories has a cluster algebra structure.
In particular, $\Cone$ provides a monoidal categorification of a cluster algebra of finite cluster type, whose classification is given in \cite{MR2004457}.

In response to the questions above, we propose the following conjecture in this paper:
\begin{conjecture*}[Conjecture \ref{conj:main1}]
Let $\levk \ge 2$ be a positive integer, $\ellk = \levk-1$ and let $\phik : K_0(\Cellk) \to K_0(\fusk)$ be the restriction of $\phik$ on $K_0(\Cellk)$.
For every $V$ in $\Cellk$, $\phik([V])$ is positive.
\end{conjecture*}
For a precise formulation of the conjecture, we make a slight adjustment of the definition of $\Cell$; see \secref{sec:Cell} and \remref{rmk:Cell}.

\textbf{Main results.}
The main results of this paper are as follows:
We prove that if the KR modules in the initial seed of $K_0(\Cellk)$ have positive quantum dimensions at level $\levk$, then all objects in \(\Cellk\) also have positive quantum dimensions at level $\levk$; see \thmref{thm:qdim}.
This is consistent with a quantum dimension version of our conjecture, and provides a strong evidence to our main conjecture.
We also establish our main conjecture, the Verlinde positivity, in certain special cases, specifically for
\(\Cell\) when the underlying simple Lie algebra \(\gzero\) is of type \(A_1\), and for \(\Cone\) when \(\gzero\) is of type \(A_n\) or one of the exceptional types \(E_6\), \(E_7\), or \(E_8\) for certain choices of height functions:
see \thmref{thm:Cone}.
In addition, we provide a detailed conjectural description for type \(D_n\).

This paper is organized as follows. 
In \secref{sec:background}, we review the basic concepts of cluster algebras, Kac-Moody algebras, and the quantum Kac-Moody algebras. We also provide a brief introduction to the Verlinde ring associated with an affine Lie algebra and a positive integer $\levk \ge 1$.
In \secref{sec:Cell}, we define the subcategories \(\Cell\), parametrized by positive integers $\ell$, of the category of finite-dimensional representations over quantum affine algebras, equipped with a cluster algebra structure.
In \secref{sec:main_conj}, we present our main conjecture concerning positivity in the Verlinde ring under a certain ring homomorphism from \(K_0(\Cellk)\). As supporting evidence, we prove its quantum dimension version under the assumption that the conjecture of Kirillov, Kuniba, Nakanishi and Suzuki regarding the positivity of the quantum dimension of certain KR modules holds.
In \secref{sec:pos_fit_cls}, we establish our main conjecture for \(\Cone\) in some special cases.
In \secref{sec:examples}, we provide detailed examples illustrating our conjecture, focusing on a non-real module and a multiply-laced type, offering additional evidence in support of our conjecture.

\subsection*{Acknowledgements}
C.-h. Lee is supported by a KIAS Individual Grant (SP067302) via the June E Huh Center for Mathematical Challenges at Korea Institute for Advanced Study.
JR Li is supported by the Austrian Science Fund (FWF): P-34602, Grant DOI: 10.55776/P34602, and PAT 9039323, Grant-DOI 10.55776/PAT9039323.
E. Park is supported by the National Research Foundation of Korea (NRF) Grant funded by the Korea Government(MSIT)(RS-2023-00273425 and NRF-2020R1A5A1016126).
We thank the KIAS Center for Advanced Computation and the Vienna Scientific Cluster (VSC) for providing access to computing resources.
\vskip 1em

\section{Background}\label{sec:background}
Throughout this paper, for a statement $P$, we set $\delta(P)$ to be $1$ or $0$ depending on whether $P$ is true or not. In particular, we set $\delta_{i,j}=\delta(i =j)$.

\subsection{Cluster algebras}
Let $\caI = \caIe \sqcup \caIf$ be a (finite or countably many) index set that decomposes into the set $\caIe$ of \emph{exchangeable} indices and the set $\caIf$ of \emph{frozen} indices. An \emph{exchange matrix} $\caB = (b_{i,j})_{i\in \caI, j\in \caIe}$ is an integer-valued matrix such that 
\bna
\item for each $j\in \caIe$, there are finitely many $i\in \caI$ such that $b_{i,j}$ are nonzero,
\item  the \emph{principal part} $(b_{i,j})_{i\in \caIe, j\in \caIe}$ of $\caB$ is skew-symmetric.
\ee
For convenience, when $i\in \caI$ and $j\in \caIf$, we define 
\[
b_{i,j} = 
\begin{cases}
-b_{j,i} & \text{ if $i\in \caIe$ and $j\in  \caIf$},   \\
0 & \text{ if $i, j\in \caIf$}.
\end{cases}
\] 
Note that, to an exchange matrix $\caB$, one can associate a quiver $Q_\caB$ as follows: the set of vertices of $Q_\caB$ is $\caI$ and the number of arrows from $i\in \caI$ to $j \in \caI$ is defined to be the number $\max\{ 0, b_{i,j} \}$. 
The quiver $Q_\caB$ also determines the matrix $\caB$ by taking 
\[
b_{i,j} := \text{(the number of arrows $i \rightarrow j$)} - \text{(the number of arrows $j \rightarrow i$)}.
\]

Let $A$ be a commutative algebra which is contained in  a field $F$. A pair $ \seed = ( \{ x_i \}_{i\in \caI}, \caB)$ is called a \emph{seed} in $A$ if the set $\{ x_i\}_{i\in \caI} $ is contained in $ A$ and the homomorphism from a polynomial ring $ \Z[\mathrm{X}_i\ ; i\in \caI] $ to $A$ defined by $\mathrm{X}_i \mapsto x_i$ ($i\in \caI$) is injective. The set $\{ x_i\}_{i\in \caI}$ is called a \emph{cluster} of $\seed$ and the elements $x_i$ are called \emph{cluster variables}. For any $\bfa = (a_i)_{i\in \caI} \in \Znat^{\caI}$, the monomial
\[
x^\bfa := \prod_{i\in \caI} x_i^{a_i}
\]
is called a \emph{cluster monomial}.

For each $k\in \caIe$, we define a set $\mu_k ( \{ x_i \}_{i\in \caI} ) = \{ y_i \}_{i\in \caI} $ and a matrix $\mu_{k}(\caB) = (c_{i,j})_{i\in \caI, j\in \caIe}$ by
\begin{align*}
y_i &:= 
\begin{cases}
(x^{\bfa'} + x^{\bfa''})/x_k  & \text{ if $i=k$,}   \\
x_i & \text{ if $i\ne k$,}
\end{cases}\\
c_{i,j} &:= 
\begin{cases}
-b_{i,j} & \text{ if $i=k$ or $j=k$},   \\
b_{i,j} + (-1)^{\delta(b_{i,k} < 0)} \max \{ b_{i,k} b_{k,j}, 0 \} & \text{ otherwise,}
\end{cases} 
\end{align*}
where $\bfa' = (a_i')_{i\in \caI}$ and $\bfa'' = (a_i'')_{i\in \caI}$ are given by 
\[
a_i' = 
\begin{cases}
0  & \text{ if $i=k$,}   \\
\max\{ 0, b_{i,k} \} & \text{ if $i\ne k$,}
\end{cases}
\qquad 
a_i'' = 
\begin{cases}
0  & \text{ if $i=k$,}   \\
\max\{ 0, -b_{i,k} \} & \text{ if $i\ne k$.}
\end{cases}
\]
The pair $\mu_k(\seed) := ( \mu_k ( \{ x_i \}_{i\in \caI} ), \mu_{k}(\caB))$ is called the \emph{mutation} of $\seed$ at $k$, 
and the following equality in $F$ is called an \emph{exchange relation}:
\[
x_k y_k  = x^{\bfa'} + x^{\bfa''}.
\]

The \emph{cluster algebra} $\CA(\seed)$, associated to an initial seed $\seed$, is the $\Z$-subalgebra of the field $F$ generated by all cluster variables obtained from $\seed$ by all possible successive mutations.
A commutative algebra $A$ has a \emph{cluster algebra structure} associated to a seed $\seed$ in $A$ if there exists a family $\mathfrak{F}$ of seeds in $A$ such that 
\bna
\item $\CA(\seed)$ coincides with $A$,
\item $\seed\in \mathfrak{F}$, and any mutation of a seed in $\mathfrak{F}$ is contained in $\mathfrak{F}$,
\item for any $\seed', \seed'' \in \mathfrak{F}$, $\seed'$ can be obtained from $\seed''$ by a finite sequence of mutations.
\ee

A cluster algebra is said to be of \emph{finite cluster type} if it has a finite number of cluster variables.
The \emph{Cartan counterpart} $C(\caB) = (c_{i,j})_{i,j\in \caIe}$ associated with an exchange matrix $\caB$ is the square matrix defined by 
$$
c_{i,j} = 
\begin{cases}
2 & \text{ if } i=j,\\
-|b_{i,j}| & \text{ if } i\ne j.
\end{cases}
$$
We then have a couple of results on cluster algebras of finite cluster types, which will be needed later.
\begin{theorem}[Finite type classification {\cite[Theorem 1.4]{MR2004457}}] 
The cluster algebra $A$ is of finite cluster type if and only if there is an exchange matrix $\caB$ of $A$ such that the Cartan counterpart $C(\caB)$ is a block diagonal matrix of Cartan matrices of finite types, up to conjugation by a permutation matrix.
\end{theorem}

\begin{theorem} [{\cite[Theorem 4.3.1]{ICT21}}] \label{Thm: cm basis}
Let $A$ be a cluster algebra of finite cluster type. The cluster monomials of $A$ form a linear basis of $A$.
\end{theorem}


\subsection{Affine Lie algebras and quantum affine algebras}
This subsection provides a brief overview of Kac-Moody algebras and their quantum counterparts, with a particular focus on the affine case.
For detailed information on Kac-Moody algebras, refer to \cite{MR1104219}.
Let $ (\cmA, \wlP, \Pi, \wlPdual, \Pi^\vee)$ be a \emph{Cartan datum} consisting of a \emph{symmetrizable generalized Cartan matrix} $\cmA = (a_{i,j})_{i,j\in I}$,  a free abelian group $\wlP$, called the \emph{weight lattice}, its dual $\wlPdual:=\Hom(\wlP, \Z)$, the set $\Pi=\{\al_i \}_{i\in I}\subset \wlP$ of \emph{simple roots} and the set $\Pi^\vee =\{h_i \}_{i\in I}\subset \wlPdual$ of \emph{simple coroots}.
They meet the following requirements:
\begin{itemize}
    \item Both \(\Pi\) and \(\Pi^\vee\) are linearly independent.
    \item The condition \(\langle h_i, \alpha_j \rangle = a_{i,j}\) holds for all \(i, j \in I\).
    \item \(|I| - \rk \cmA = \rk_{\Z} \wlP - |I|.\)
\end{itemize}
Here, \(\langle \cdot, \cdot \rangle\) denotes the canonical pairing \(\wlP \times \wlPdual \to \Z\) and \(\wlPdual \times \wlP \to \Z\).
We define \(\rlQ := \bigoplus_{i \in I} \Z \alpha_i\), known as the \emph{root lattice}. 
Additionally, we introduce the complex vector spaces \(\h := \wlPdual \otimes \mathbb{C}\) and \(\hdual := \wlP \otimes \mathbb{C}\).
We call $\Pplus: = \{\la \in \wlP \mid \langle \la, h_i\rangle \in \Znat \text{ for all } i \in I\}$ the set of \emph{dominant integral weights}.

With this datum, we define an associated Lie algebra $\g$, known as the \emph{Kac-Moody algebra}, which is denoted by \(\g(\cmA)\) in \cite{MR1104219}. Under the assumption that $\cmA$ is symmetrizable, \(\g\) can be described by the following presentation; see \cite[Theorem 9.11]{MR1104219}.
\begin{definition} 
The {\em Kac-Moody algebra} $\g$ associated with  $(\cmA,\wl,\Pi,\wl^\vee,\Pi^\vee)$ is the Lie algebra 
over $\C$ generated by $e_i,f_i$ $(i \in I)$ and $h \in  \wl^{\vee}$ satisfying following relations:
\bnum
\item \( [h, h'] = 0 \quad \) for \( h, h' \in \wl^{\vee}, \)
\item \( [e_i, f_j] = \delta_{i,j} h_i \quad \text{for } i, j \in I, \)
\item \( [h, e_i] = \langle h, \alpha_i \rangle e_i, \quad [h, f_i] = -\langle h, \alpha_i \rangle f_i \quad \text{for } i\in I, h \in \wl^{\vee}, \)
\item \( (\ad e_i)^{1 - a_{i,j}} e_j = (\ad f_i)^{1 - a_{i,j}} f_j = 0 \quad \text{for } i \ne j. \)
\ee
\end{definition}

There exists a $\Q$-valued, non-degenerate symmetric bilinear form \((\cdot | \cdot)\) on $\wlP$ satisfying
\((\al_i| \al_i)>0\) for all \( i \in I \), and for any $\la \in \wlP$,
\[
\langle h_i, \la \rangle = \frac{2 (\al_i| \la)}{(\al_i| \al_i)}.
\]
The \emph{Weyl group} \( W \) of \(\g\) is the subgroup of \(\Aut(\wlP)\) generated by the \emph{fundamental reflections} \( s_i \) for \( i \in I \), where
\begin{equation}
s_{i}(\la) = \la - \left\langle h_{i}, \la \right\rangle \alpha_{i} \quad \text{for } \la \in \wlP.
\label{eq:weyl_si}
\end{equation}
It turns out that $W$ is a Coxeter group, and the bilinear form \((\cdot | \cdot)\) on $\wlP$ is $W$-invariant.

Now we turn to a quantum deformation of the universal enveloping algebra $U(\g)$ of a Kac-Moody algebra $\g$.
Let $q$ be an indeterminate and $\cor$ the algebraic closure of the subfield $\C(q)$
in $\corh\seteq\bigcup_{m >0}\C((q^{1/m}))$. 
For each $i\in I$, set $q_i = q^{(\alpha_i\mid \alpha_i)/2}$, and for any $m,n\in \Znat$, define
\begin{equation*}
 \begin{aligned}
 \ &[n]_i =\frac{ q^n_{i} - q^{-n}_{i} }{ q_{i} - q^{-1}_{i} }, \quad 
 \ &[n]_i! = \prod^{n}_{k=1} [k]_i , \quad 
 \ &\left[\begin{matrix}m \\ n\\ \end{matrix} \right]_i=  \frac{ [m]_i! }{[m-n]_i! [n]_i! }.
 \end{aligned}
\end{equation*}
We set $\mathsf{d}$ to be the smallest positive integer such that 
$\mathsf{d} \frac{(\al_i| \al_i)}{2}\in\Z$  for all $i\in I$.
Note that $\mathsf{d}=1$ if $(\al_i | \al_i) \in 2\Z_{>0}$ for all $i\in I$.

\begin{definition} \label{Def: GKM}
The {\em quantum group} $\uqg$ associated with  $(\cmA,\wl,\Pi,\wl^\vee,\Pi^\vee)$ is the associative algebra over $\cor$ generated by $e_i,f_i$ $(i \in I)$ and
$q^{h}$ $(h \in  \mathsf{d}^{-1} \wl^{\vee})$ satisfying following relations:
\bnum
\item  $q^0=1, \  q^{h} q^{h'}=q^{h+h'} $\quad for $ h,h' \in \mathsf{d}^{-1} \wl^{\vee},$
\item  $q^{h}e_i q^{-h}= q^{\langle h, \alpha_i \rangle} e_i$,
$q^{h}f_i q^{-h} = q^{-\langle h, \alpha_i \rangle }f_i$\quad for $h \in \mathsf{d}^{-1}\wl^{\vee}, i \in I$,
\item  $e_if_j - f_je_i =  \delta_{i,j} \dfrac{K_i -K^{-1}_i}{q_i- q^{-1}_i }, \ \ \text{ where } K_i=q_i^{ h_i},$
\item  $\displaystyle \sum^{1-a_{ij}}_{k=0}
(-1)^ke^{(1-a_{ij}-k)}_i e_j e^{(k)}_i =  \sum^{1-a_{ij}}_{k=0} (-1)^k
f^{(1-a_{ij}-k)}_i f_jf^{(k)}_i=0 \quad \text{ for }  i \ne j, $
\ee
where 
$e_i^{(k)}=e_i^k/[k]_i!$ and $f_i^{(k)}=f_i^k/[k]_i!$.
\end{definition}

For a given Cartan datum, there is a category \(\RO\) (resp. \(\RO^q\)) of \emph{integrable representations} for \(\g\) (resp. \(\uqg\)).
For each dominant integral weight \(\la \in \Pplus\), there exists a corresponding irreducible highest weight representation \(L(\la)\) (resp. \(L^q(\la)\)) with highest weight \(\la\). Each such module belongs to \(\RO\) (resp. \(\RO^q\)), and every irreducible module in the category is isomorphic to one of these modules. Moreover, every object in \(\RO\) (resp. \(\RO^q\)) is isomorphic to a direct sum of these irreducible highest weight representations.

From now on, let $\cmA$ be a generalized Cartan matrix of untwisted affine type $X^{(1)}_{\rank}$, where $X=A,B,\dots, G$. Specifically, its Dynkin diagram corresponds to one of those listed in \cite[Table Aff 1]{MR1104219}. 
This diagram is obtained by adding an extra vertex to the diagram of finite type $X_{\rank}$, as shown in \cite[Table Fin]{MR1104219}, and is commonly referred to as the \emph{extended Dynkin diagram}.
For the index set, we write \( I = \{0\} \sqcup I_0 \), where $0\in I$ corresponds to the added vertex in the extended Dynkin diagram.

Note that \( \cmA \) has rank \( N \), so its corank is 1. There is a sequence of integers \((\sfa_i)_{i \in I}\), often referred to as \emph{marks} in the literature, which satisfy
\[
\sum_{j \in I} a_{i,j} \sfa_j = 0 \quad \text{for all } i \in I,
\]
with the normalization \(\sfa_0 = 1\). These values are numerical labels provided in \cite[Table Aff 1]{MR1104219}.
Similarly, the \emph{comarks} \((\cosfa_i)_{i \in I}\) are a sequence of integers that satisfy
\[
\sum_{j \in I} a_{j,i} \cosfa_j = 0 \quad \text{for all } i \in I,
\]
with the normalization \(\cosfa_0 = 1\).
The numbers \( h = \sum_{i \in I} \sfa_i \) and \( \hvee = \sum_{i \in I} \cosfa_i \) are called the \emph{Coxeter number} and the \emph{dual Coxeter number}, respectively.

Let us fix a Cartan datum as follows.
Let $\wlPdual$ be a free $\Z$-module with basis $\{h_0,\dots, h_{\rank}, \der\}$, and then
define the dual lattice $\wlP: = \Hom(\wlPdual, \Z)$.
For each $i\in I$, we define the \emph{fundamental weights} $\La_i\in \wlP$ as the linear functionals satisfying
$\langle \La_i, h_j \rangle=\delta_{i,j}$ for $j\in I$ and $\langle \La_i, \der\rangle = 0$.
We then choose $\nullroot\in \wlP$ such that 
\begin{equation}
\langle \nullroot, h_i \rangle = 0 \quad \text{for all } i \in I,
\label{eq:delta_hi}
\end{equation}
and $\langle \nullroot, \der \rangle=1$.
In other words, the ordered set $\{\La_0, \dots, \La_{\rank},\nullroot\}$ is a $\Z$-basis of $\wlP$, dual to the ordered basis $\{h_0,\dots, h_{\rank}, \der\}$ of $\wlPdual$.
For each \(i \in I\), define \(\alpha_i \in \wlP\) so that \(\langle \alpha_j, h_i \rangle = a_{i,j}\) for any \(j \in I\), along with \(\langle \alpha_0, \der \rangle = 1\) and \(\langle \alpha_i, \der \rangle = 0\) for \(i \in I_0\).
Specifically, we have
\begin{equation}
\alpha_0 = \nullroot + \sum_{i \in I} a_{i,0} \La_{i}, \quad \alpha_j = \sum_{i \in I} a_{i,j} \La_{i} \quad \text{for } j \in I_0.
\label{eq:alpha_in_La}
\end{equation}
With these choices, \((\cmA, \wlP, \Pi, \wlPdual, \Pi^\vee)\) forms a Cartan datum.
The relation $\nullroot = \sum_{i\in I}\sfa_i \alpha_i \in \wlP$ holds, and $\nullroot$ is referred to as the \emph{imaginary root}.
Additionally, we define the \emph{canonical central element} $\cent := \sum_{i\in I}\cosfa_i h_i \in \wlPdual$.

Let \(\g := \g(\cmA)\) denote the affine Kac-Moody algebra associated with the Cartan datum \((\cmA, \wlP, \Pi, \wlPdual, \Pi^\vee)\), referred to as the \emph{affine Lie algebra}.
This datum also gives rise to a corresponding \emph{quantum affine algebra} $\uqg$.
In this paper, we will also use the notation $\ghat$ for $\g$, which is common in the literature, while $\g_0$ will be discussed in the subsequent paragraphs.

We define the \emph{level} of \(\La \in \wlP\) as the number
\begin{equation}\label{eq:level}
\levk := \langle \La, \cent \rangle = \sum_{i \in I} \cosfa_i \langle \La, h_i \rangle.
\end{equation}
The canonical central element \(\cent\in \g\) acts on \(L(\La)\) for \(\La \in \Pplus\) as a scalar equal to the level of \(\La\). This value is also referred to as the level of \(L(\La)\).

The set \( I_0 = I\setminus\{0\} \) serves as the index set for the underlying simple finite-dimensional Lie algebra \(\gzero\) of type \( X_{\rank} \).
Specifically, $\gzero$ is the Lie subalgebra of $\g$ generated by the Chevalley generators $e_i$, $f_i$  and $h_i$ for $i \in I_0$.
Analogously, $\uqgzero$ is the $\cor$-subalgebra of $\uqgzero$ generated by $e_i,f_i,K^{\pm 1}_i$ for $i \in I_0$.
They can also be described as a Kac-Moody algebra and its quantum counterpart as follows.
Let $\wlPdual_0$ denote the $\Z$-submodule of $\wlP$ with basis $\{h_i\}_{i\in I_0}$.
For $\la \in \wlP =\Hom(\wlPdual, \Z)$, let $\ol{\la}\in \Hom(\wlPdual_0, \Z)$ be the restriction of
$\la$ to $\wlPdual_0$.
Define
\begin{equation}
\wlP_0 := \oplus_{i\in I_0} \Z \ol{\La}_i, \quad 
\Pi_0 := \{\ol{\alpha}_i\}_{i\in I_0}, \quad
\text{and} \quad 
\Pi_0^\vee := \{h_i\}_{i\in I_0}.
\label{eq:finite_cartan}
\end{equation}
Then \(\gzero\) and \(\uqgzero\) are the Kac-Moody algebra and the corresponding quantum group associated with the Cartan datum \((\cmA_0, \wlP_0, \Pi_0, \wlPdual_0, \Pi_0^\vee)\), respectively.
Throughout paper, we will also use the notation \(\poid_i := \ol{\La}_i\) for \(i \in I_0\).

Let \(\repg\) and \(\repuqgzero\) denote the categories of integrable representations for \(\gzero\) and \(\uqgzero\), respectively. The objects in those categories are finite-dimensional.
The Grothendieck rings \(K_0(\repg)\) and \(K_0(\repuqgzero)\) are both \(\Z\)-rings, with bases \(\{[L(\la)] \mid \la \in P_0\}\), and \(\{[L^q(\la)] \mid \la \in P_0\}\), respectively.
There exists a ring isomorphism that maps \([L^q(\la)]\) to \([L(\lam)]\):
\begin{equation}\label{eq:iso_hom}
K_0(\repuqg) \xrightarrow{\cong} K_0(\repg).
\end{equation}

Throughout the paper, we will also denote the simple \(\uqgzero\)-module \(L^q(\la)\) as \(L(\la)\) whenever the context is clear.

Let $\g': = [\g, \g]$ denote the derived subalgebra of $\g$.
It is a subalgebra generated by $e_i,f_i,h_i$ for $i \in I$, and is isomorphic to
\[
(\C[t, t^{-1}]\otimes_{\C}\gzero)\oplus \C \cent,
\]
a central extension of the loop algebra, equipped with a certain Lie bracket.
This is also called the affine Lie algebra.
Analogously, we define $\uqpg$ to be the $\cor$-subalgebra of $\uqg$ generated by $e_i,f_i,K^{\pm 1}_i$ for $i \in I$.
The key distinction between the algebras \( \uqg \) and \( \uqpg \) is that \( \uqpg \) has a rich theory of finite-dimensional representations, whereas any non-trivial representation of \( \uqg \) is infinite-dimensional. 
In the literature, \( \uqpg \) is also known as the quantum affine algebra.
This paper focuses on certain finite-dimensional representations of \( \uqpg \); see \secref{sec:Cell} for additional details.

We now turn to the action of Weyl groups. We define \(\Pcl := \wlP / (\wlP \cap \mathbb{Z} \nullroot)\). 
When there is no risk of confusion, we will simply write \(\La_i\) for the image of \(\La_i\in \wlP \) under the quotient map \(\wlP \to \Pcl\).
The level of an element in \(\Pcl\) is defined in the same way as in \eqref{eq:level}.
Note that, from \eqref{eq:weyl_si} and \eqref{eq:delta_hi}, we have $w(\nullroot) = \nullroot$ for all $w\in W$.
Then, the \emph{affine Weyl group} $W$ acts on both $\wlP$ and $\Pcl$.
In particular, the fundamental reflection $s_i$ for $i\in I$ acts linearly on $\Pcl$ by 
\begin{equation}
s_i \La_j=\La_j -\delta_{i,j}\alpha_i \pmod{\Z \nullroot},\quad j\in I.
\label{eq:fundamental_reflection_affine}
\end{equation}
Using \eqref{eq:alpha_in_La}, we can express this action entirely in terms of $\cmA$ and the fundamental weights.
Let $\rho:=\sum_{i\in I}\La_{i}\in \Pcl$ be the \emph{affine Weyl vector}.
The \emph{shifted Weyl group action} of $W$ on $\Pcl$ is given by
\begin{equation}\label{eq:shifted_action}
w \cdot \la := w(\la + \rho) - \rho,
\end{equation}
for $\la \in \Pcl$.
Finally, the \emph{finite Weyl group} \(W_0\), generated by \(\{s_i\}_{i \in I_0}\), acts on the weight lattice \(\wlP_0\) \eqref{eq:finite_cartan} of \(\gzero\). The fundamental reflection \(s_i\) for \(i \in I_0\) acts on \(\wlP_0\) by
\begin{equation}
s_i \poid_j = \poid_j - \delta_{i,j} \ol{\alpha}_i,\quad j\in I_0
\label{eq:fundamental_reflection_finite}
\end{equation}
which can be expressed explicitly in terms of $\cmA_0$ and \(\{\poid_i\}_{i \in I_0}\).

\subsection{Verlinde rings}\label{subsec:verlinde}
This subsection provides a brief introduction to the Verlinde ring associated to an affine Lie algebra $\g$ and an integer $\levk \ge 1$.
For a detailed systematic account, refer to \cite{MR4327093}. 
For additional insight from a conformal field theory perspective, see \cite[Chapters 14 and 16]{MR1424041}.

Let \(\Pk\) denote the subset of \(\Pcl\) consisting of elements with level \(\levk\).
We define 
\[
\Pkp := \{\la \in \Pk \mid \la=\sum_{i\in I}c_{i}\La_{i},\,c_{i}\in \Znat\}.
\]
An element $\la=\sum_{i\in I}c_{i}\La_{i}\in \Pcl$ belongs to $\Pkp$ if and only if
\[
\sum_{i\in I}\cosfa_{i}c_{i}=\levk,\quad c_i\in \Znat.
\]
Hence, $\Pkp$ is a finite set. 
Note that there is a bijection $\Pk\to \wlP_0$ given by the restriction $\la \mapsto \ol{\la}$:
\begin{equation}
\la = \sum_{i \in I} c_i \La_i \mapsto \overline{\la} = \sum_{i \in I_0} c_i \poid_i.
\label{eq:wt_proj}
\end{equation}

For a given $\poid\in \wlP_0$, we can find a unique element $\la \in \Pk$ such that $\ol{\la} = \poid$.
Under this bijection, $\Pkp$ corresponds to the dominant integral weights in the fundamental alcove:
\[
\{\poid \in \wlP_0 \mid \langle h_i, \poid \rangle \ge 0 \text{ for each } i \in I_0 \text{ and } \langle \theta^{\vee}, \poid \rangle \le \levk\},
\]
where $\theta^{\vee} := \sum_{i \in I_0} \cosfa_i h_i$.
In this subsection, we adopt a specific normalization for a non-degenerate symmetric bilinear form \((\cdot| \cdot)_0\) on \(\wlP_0\) so that \((\theta| \theta)_0 = 2\), where \(\theta = \sum_{i \in I_0} \sfa_i \ol{\alpha}_i\) is the highest root of \(\gzero\).

Given a pair $\la, \mu \in \Pk$, define
\begin{equation}\label{eq:modularS}
S_{\la,\mu}:=C(\cmA,\levk) \cdot \sum_{w\in W_0} (-1)^{\ell(w)}\exp \left(-{\frac{2\pi \imag ( w(\ol{\la}+\ol{\rho})| \ol{\mu}+\ol{\rho})_0}{\levk+\hvee}}\right),
\end{equation}
where $C(\cmA,\levk)$ is a certain normalizing constant that depends only on the Cartan datum and not on $\la$ and $\mu$; see \cite[Theorem 13.8]{MR1104219} for an explicit expression.
Note that the sum is taken over the finite Weyl group $W_0$.
The matrix \(S = (S_{\la, \mu})_{\la, \mu \in \Pkp}\) is referred to as the \emph{modular \(S\)-matrix}. 
This quantity plays a role, for example, in describing the modular transformation properties of characters of integrable highest weight representations at level \(\levk\) for affine Lie algebras; see \cite{MR750341} and \cite[Theorem 13.8]{MR1104219}. This connection explains its name.

Here, we provide a summary of several key properties of the modular \(S\)-matrix.

(1). Symmetry: The modular \(S\)-matrix is symmetric, so for \(\la, \mu \in \Pk\),
\[
S_{\la,\mu} = S_{\mu,\la}.
\]

(2). Unitarity: The modular \(S\)-matrix is unitary:
\begin{equation}\label{eq:unitarity}
SS^{\dagger} = I_n,
\end{equation}
where \(S^{\dagger}\) denotes the transpose of the complex conjugate of \(S\), and \(I_n\) is the identity matrix of size \(n = |\Pkp|\).

(3). Weyl group symmetry: For \(w \in W\) and \(\la, \mu \in \Pk\), the following relations holds:
\begin{equation}\label{eq:shiftWS}
S_{w \cdot \la, \mu} = (-1)^{\ell(w)} S_{\la, \mu}.
\end{equation}
Here, the shifted action of $W$ defined in \eqref{eq:shifted_action} is used.

(4). Conjugation symmetry: For \(\la, \mu \in \Pk\), we have
\begin{equation}\label{eq:longestconj}
S_{\la^{*}, \mu} = S^{*}_{\la, \mu},
\end{equation}
where $S^{*}_{\la, \mu}$ is the complex conjugate of $S_{\la, \mu}$ and $\la^*\in \Pk$ is defined as follows:
For \(\poid \in \wlP_0\), let \(\poid^{*} := -w_0 \poid \in \wlP_0\), where \(w_0\) is the longest element of the finite Weyl group \(W_0\).
For an affine weight \(\la \in \Pk\), its conjugate \(\la^{*} \in \Pk\) is the unique element of \(\Pk\) such that \(\ol{\la^{*}} = \ol{\la}^{*}\). 

The simple objects in the category $\fusk$ of integrable representations of level $\levk$ over the affine Lie algebra $\ghat'$ are indexed by $\Pkp$, and any object can be written as a direct sum of these simple objects (\cite[Theorem 10.7]{MR1104219}).
However, the tensor product of two objects in \(\fusk\), considered as representations of the Lie algebra, does not belong to \(\fusk\) because its level becomes \(2\levk\).

Surprisingly, there exists a different kind of product on $\fusk$, called the \emph{fusion product}.
The \emph{Verlinde ring} $K_0(\fusk)$, also known as the \emph{WZW fusion ring} or simply the \emph{fusion ring}, is the Grothendieck ring of the category $\fusk$ of integrable representations $\levk$ over $\ghat'$.
Specifically, it is a free $\Z$-module with basis $\{[L(\la)] :\la \in \Pkp\}$, endowed with the product
\[
[L(\la)]\cdot [L(\mu)]=\sum_{\nu \in \Pkp}\fusN[L(\nu)],
\]
where $\fusN\in \Znat$ denotes the dimension of the \emph{conformal block} $V_{\Pbb^1}(\la,\mu,\nu^{*})$ on $\Pbb^1$; see \cite{MR1360497, MR4327093}.
The coefficients \(\fusN\) are called \emph{fusion coefficients}.

It is a commutative, associative ring with unity \([L(\levk \La_0)]\).
The associativity is closely related to the factorization property of the space of conformal blocks; see \cite[Chapters 3 and 4]{MR4327093}.
The fusion coefficients can be written in terms of the modular \(S\)-matrix via the \emph{Verlinde formula}:  
\begin{equation}\label{eq:Verlinde}
\fusN = \sum_{\omega \in \Pkp} \frac{S_{\la, \omega} S_{\mu, \omega} S^{*}_{\nu, \omega}}{S_{\levk \La_0, \omega}}.
\end{equation}
Note that although the normalizing factor in \eqref{eq:modularS} has no effect on \eqref{eq:qdim} below, it is necessary for the validity of \eqref{eq:Verlinde}.

\begin{definition}\label{def:verlinde_positivity}
An element $\sum_{\la \in \Pkp}c_{\la}[L(\la)]\in K_0(\fusk)$ is \emph{non-negative} if $c_{\la}\ge 0$ for all $\la \in \Pkp$.
A non-zero element of $K_0(\fusk)$ is \emph{positive} if it is non-negative.
\end{definition}

Although the fusion product has its origins in conformal field theory \cite{MR954762} and is mathematically defined in a rather complicated algebro-geometric way, there is a way to handle it using only elements from Lie theory, while avoiding the need for explicit computation of the $S$-matrix, which is computationally demanding.
Specifically, there exists a surjective ring homomorphism $\pik : K_0(\repg)\to K_0(\fusk)$.
We provide a description of this map below.
The Grothendieck ring $K_0(\repg)$ of finite-dimensional representations of $\gzero$ is a free $\Z$-module with basis $\{[L(\poid)] \mid \poid \in \wlP_{0,+}\}$.
Thus, it is sufficient to specify
$\pik([L(\poid)])$ to define a ring homomorphism.

For a given \(\la \in \Pk\), either:  
\begin{itemize}
\item There is no non-trivial element \(w \in W\) such that \(w \cdot \la = \la\),  
\item or there exists a non-trivial element \(w \in W\) such that \(w \cdot \la = \la\).
\end{itemize}
We define \(\la \in \Pk\) as \emph{null} if the latter condition holds.
If the first condition holds, then there exists a unique $\la'\in \Pkp$ such that
\begin{equation}\label{eq:alcoverep}
\la'=w\cdot \la
\end{equation}
for some unique $w\in W$. 
Then $\pik : K_0(\repg)\to K_0(\fusk)$ can be defined as follows:
\begin{equation}\label{eq:Valcove}
\pik([L(\poid)]):=
\begin{cases} 
0 & \text{if $\la$ is null,}\\ 
(-1)^{\ell(w)}[L(\la')] & \text{otherwise},
\end{cases}
\end{equation}
where $\poid\in \wlP_{0,+}$, and $\la \in \Pk$ is the unique affine weight of level $\levk$ such that $\ol{\la} = \poid$, which is not necessarily in $\Pkp$.
In case $\la$ is not null, $\la'\in \Pkp$ and $w\in W$ are given as in \eqref{eq:alcoverep}.
The existence of this homomorphism is established in \cite{MR1360497}; 
see also \cite[Theorem 4.2.9]{MR4327093} and the references therein for further details.
This enables expressing the fusion product in terms of the tensor product in \(\repg\).
For \(\La_1, \La_2 \in \Pkp\), the product \([L(\La_1)] \cdot [L(\La_2)]\) is given by:
\[
[L(\La_1)] \cdot [L(\La_2)] = \pik\left([L(\overline{\La_1})]\right) \cdot \pik\left([L(\overline{\La_2})]\right) = \pik\left([L(\overline{\La_1}) \otimes L(\overline{\La_2})]\right).
\]

Simply put, determining the image of \(\pik\) involves moving a weight from the fundamental chamber to the fundamental alcove via the generators of $W$ while keeping track of the number of simple reflections applied.
This procedure is demonstrated in \exref{ex:alcove} and illustrated in \figref{fig:alcove_image}.

\begin{figure}[ht!]
    \centering
    \includegraphics[width=0.7\textwidth]{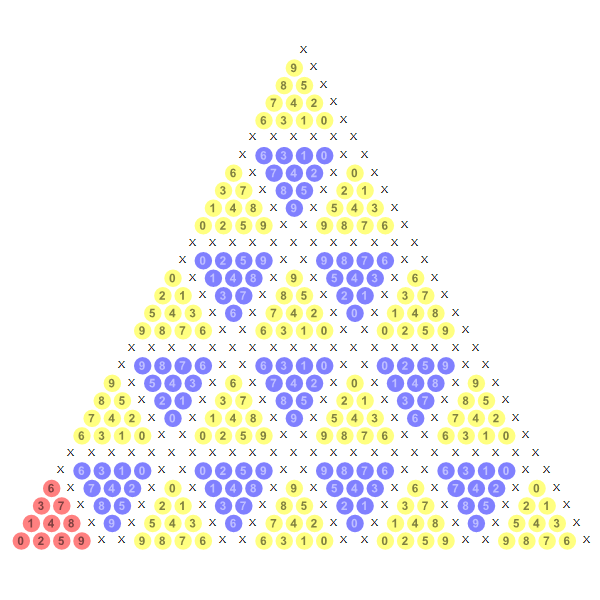}
    \caption{Visualization of \(\pik\); see \exref{ex:alcove}.}
    \label{fig:alcove_image}
\end{figure}

\begin{example}\label{ex:alcove}
Let \(\gzero\) be of type \(A_2\).  
The elements of \(\wlP_0\) can be represented as points in a hexagonal lattice in \(\mathbb{R}^2\), with basis vectors \(\poid_1 = (1, 0)\) and \(\poid_2 = \left(\frac{1}{2}, \frac{\sqrt{3}}{2}\right)\).  

Consider now \(\levk = 3\).  
Each element of \(\Pk\) is uniquely associated with a point in the hexagonal lattice through the bijection given in \eqref{eq:wt_proj}.
Among these, 10 elements belong to \(\Pkp\), with each \(\la \in \Pkp\) corresponding to a bottom-left lattice point marked in red in \figref{fig:alcove_image}.

For \(\poid = c_1\poid_1 + c_2\poid_2\) with \(c_1, c_2 \in \Z_{\ge 0}\) and the corresponding \(\la = (\levk - c_1 - c_2)\La_0 + c_1\La_1 + c_2\La_2\), \(\la\) is null if \(\poid\) is located at a lattice point crossed by \(\times\).  
Otherwise, $\poid$ is assigned a specific numbering, and \(\ol{\la'}\) in \eqref{eq:alcoverep} corresponds to a red point with the same numbering as \(\poid\).  
If \(\poid\) is marked in blue, then \(\pik(L(\poid)) = -L(\la')\); otherwise, \(\pik(L(\poid)) = L(\la')\).
\end{example}

For $\la \in \Pk$, we define
\begin{equation} \label{eq:qdim}
\ddD_{\la} : =
\frac{S_{\la, \hatzero}}{S_{\hatzero, \hatzero}}
=\frac{\prod_{\alpha\in \Delta_{+}}\sin \frac{\pi(\ol{\la}+\ol{\rho}|\alpha)_0}{\levk+\hvee}}{\prod_{\alpha\in \Delta_{+}}\sin \frac{\pi (\ol{\rho}| \alpha)_0}{\levk+\hvee}},
\end{equation}
where $\Delta_{+}$ denotes the set of positive roots for $\gzero$, and the last equality is a consequence of the Weyl denominator formula.
Especially, for $\la \in \Pkp$, we have $\ddD_{\la}>0$, referring to it as the \emph{quantum dimension} of $L(\la)$, which we denote by
$\qdimk L(\la)$.
If $\la \in \Pk$ is null, then $\ddD_{\la}=0$.
It turns out that $\qdimk:K_0(\fusk) \to \RR$, extended linearly from the values of $\qdimk L(\la)$ for $\la \in \Pkp$, is a ring homomorphism.
This can be seen from the Verlinde formula \eqref{eq:Verlinde}, and the unitarity \eqref{eq:unitarity} of the $S$-matrix.
Furthermore, any positive element in $K_0(\fusk)$ has a positive quantum dimension.

\section{The categories $\Cell$ over quantum affine algebras} \label{sec:Cell}
Let $\g$ be an affine Lie algebra of untwisted type.
In this subsection, we consider a certain subcategory of the category of finite-dimensional integrable $\uqpg$-modules, which was introduced by Hernandez and Leclerc in \cite{MR2682185,MR3500832} from the viewpoint of cluster algebras.

Let $\gf$ be the simple Lie algebra of finite simply-laced type associated with $\uqpg$ (see \cite{KKOP22} and \cite[Section 6.1]{KKOP24}), $\cmA_{{\fin}} $ the Cartan matrix of $\gf$ and $I_\fin$ the index set of $\cmA_{{\fin}}$. 
A $Q$-datum $\mathcal{Q} = (\Dynkin, \sigma, \xi) $ is a triple consisting of the \emph{Dynkin diagram} $\Dynkin$ of $\gf$, the \emph{folding automorphism} $\sigma$ between $\gf$ and $\gzero$,  and a \emph{height function} $\xi\cl \If \rightarrow \Z$ (see \cite{FO21} and refer to \cite[Section 6.1]{KKOP24} for precise definitions). 
The folding automorphism $\sigma$ induces a surjection $\pi\cl \If \twoheadrightarrow I_0$. Define
\[
\sigma^\circ(\g) \seteq \{  (\im, p) \in \If \times \Z \mid p-\xi(\im) \in 2 d_\im \Z  \},
\]
where $d_\im = | \pi^{-1} (\pi(\im))| $. 
For $ i \in I_0 $, we set 
\begin{equation}
d_i := |\pi^{-1}(i)|
\label{eq:d_i}
\end{equation}
so that $d_\im = d_{\pi(\im)}$.

We denote by $\Cg$ the category of finite-dimensional integrable $\uqpg$-modules.
It is known that the isomorphism classes of simple modules in $\Cg$ are parameterized by the set $(1+z\cor[z])^{I_0}$ of $I_0$-tuples of monic polynomials in $z$, which are called \emph{Drinfeld polynomials} (\cite{MR1357195, MR1300632}).
They form a $\Z$-basis of the Grothendieck ring $K_0(\Cg)$.
Because $\uqgzero$ is a subalgebra of $\uqpg$, any object in $\Cg$ also becomes a $\uqgzero$-module by restriction, inducing a ring homomorphism
\begin{equation}\label{eq:res_hom}
\res : K_0(\Cg) \to K_0(\repuqgzero).
\end{equation}
Let $\mathcal{Y} := \Z[Y_{\pi(\im),p}^{\pm 1} \mid (\im,p) \in \sigma^\circ(\g)]$ be the Laurent polynomial ring in formal variables $Y_{\pi(\im),p}$.
We write a monomial $m \in \mathcal{Y}$ as 
\begin{align*}
m = \prod_{(\im,p) \in \sigma^\circ(\g)} Y_{\pi(\im),p}^{u_{\pi(\im),p}(m)}.
\end{align*}
A monomial $m \in \mathcal{Y}$ is \emph{dominant} if $u_{\pi(\im),p}(m) \ge 0$ for all $(\im,p) \in \sigma^\circ(\g)$.
For a dominant monomial $m$, there is a simple module $L(m) \in \Cg$ associated with the Drinfeld polynomial $(\prod_p (1-q^pz)^{u_{\pi(\im),p}(m)})_{\pi(\im) \in I_0}$.
For any $i\in I_0$ and $p\in \Z$, we write $L(Y_{i,p})$ for
the \emph{$i$-th fundamental weight module} with \emph{spectral parameter} $q^p$.
For each $(i,m,p) \in I_0\times \Znat\times \Z$, let $\Wimp$ be the simple $\uqghatp$-module corresponding to the monomial
$Y_{i,p}Y_{i,p+2d_i} \cdots Y_{i,p+2(m-1)d_i}$.
We call $\Wimp$ a \emph{Kirillov-Reshetikhin module} (hereafter referred to as the KR module).

We define $\catCO$ to be the smallest full subcategory of $\Cg$ such that 
\bna
\item $\catCO$ contains $L(Y_{\pi(\im), p})$ for all $(\im, p) \in \sigma^\circ(\g)$,
\item $\catCO$ is stable under taking subquotients, extensions and tensor products. 
\ee
The category $\catCO$ is called the \emph{Hernandez-Leclerc} category. For $\ell \in \Znat$, we define 
\begin{align} \label{eq:sigma ell for C ell}
\sigma_\ell(\g) \seteq \{  (\im, p) \in \sigma^\circ(\g) \mid  \xi(\im) - 2d\ell - 2(d-1) \le p \le \xi(\im) \},
\end{align}
where $d = \max_{i\in I_0}\{ d_i \}$.
We define $\Cell$ to be the smallest full subcategory of $\catCO$ such that 
\bna
\item $\Cell$ contains $L(Y_{\pi(\im), p})$ for all $(\im,p) \in \sigma_\ell(\g)$,
\item $\Cell$ is stable under taking subquotients, extensions and tensor products. 
\ee

\begin{remark}\label{rmk:Cell}
The definition of $\Cell$ is slightly modified from that of the category $C_\ell$ introduced by Hernandez-Leclerc in \cite[Definition 3.1]{MR2682185} (and see also \cite{HL21}) for the purpose of the paper. When $\g$ is of simply-laced affine type and $\xi$ is a height function arising from a bipartite quiver, the category $\Cell$ coincides with Hernandez-Leclerc's category $C_\ell$.  
\end{remark}

It is shown in \cite{MR2682185} and \cite{MR3500832} that the Grothendieck ring $K_0(\Cell)$ has a cluster algebra structure.
Let us recall the initial seed of $K_0(\Cell)$ arising from KR modules. 
Let $\caI := \sigma_\ell(\g)$. Define $\caIf := \{ (\im, p_\im) \mid \im \in \If \} \subset \caI$, where
for each $\im \in \If$, $p_\im$ is the integer such that $(\im, p_\im) \in \caI$ and $p_\im \le r$ for any $(\im, r) \in \caI$. We now set $\caIe := \caI \setminus \caIf$ so that $\caI = \caIe \sqcup \caIf$.
For any $(\im, p) \in \caI$, we set
\[
z_{\im, p} := \prod_{k\ge0,\  p+2d_\im k \le \xi(\im)} Y_{\pi(\im), p+2d_\im k}
\]
and denote the corresponding KR module by
\[
M_{\im, p} := L(z_{\im, p}) \in \Cell.
\]
We define arrows between $(\im, r), (\jmath, s) \in \caI$ as follows:
\begin{align*}
(\im, r) \longrightarrow (\jmath, s) \quad  \Longleftrightarrow \quad
 a_{\pi(\im), \pi(\jmath)} \ne 0 \text{ and } s = r+d_\im a_{\pi(\im), \pi(\jmath)} + d_\jmath - d_\im,
\end{align*}
where  $ \cmA = (a_{i, j})_{i,j\in I}$ is the affine Cartan matrix.
Letting $\caB$ be the exchange matrix obtained from the quiver, the following
\begin{align} \label{Eq: initial seed}
\seed = ( \{[M_{\im, p} ]\}_{(\im, p) \in \caI},  \caB)
\end{align}
is an initial seed of the cluster algebra $K_0(\Cell)$.
A module \(M\) in \(\Cell\) is called \emph{frozen} if \(M\) is isomorphic to a tensor product of \(M_{\im, p}\)'s $( ( \im, p) \in \caIf$).

\begin{example} \label{Ex: AB 1} \
\bni
\item Let $U_q'(\g)$ be of affine type $A_3^{(1)}$. In this case, $\gf = \g_0$ is of type $A_3$,  
$\If = I_0 = \{1,2,3 \}$, $\pi = \id$,  and  $d_i = 1$ for $i\in \If$.
Since $\pi$ is the identity, we regard $\If = I_0$ as a subset of $I$. Define a height function $\xi\cl \If \rightarrow \Z$ by $\xi (1) = -1$, $\xi (2) = 0$ and $\xi (3) = 1$, and set $\ell=2$.
 We then have 
\begin{align*}
\sigma^\circ(\g) &= \{  (1, p_1), (2, p_2), (3, p_3) \mid p_1, p_3 \in 2\Z + 1,\ p_2 \in 2\Z \}, \\
\sigma_2(\g) & = \{  (\imath, p) \in \sigma_0(\g) \mid  \xi(\im) - 4 \le p \le \xi(\im) \}. 
\end{align*}
Here $\sigma_2(\g)$ can be drawn pictorially as follows:
$$
\scalebox{0.8}{\xymatrix@C=1ex@R=  0.0ex{ 
\imath \diagdown\, p     \ar@{-}[dddd]<3ex> \ar@{-}[rrrrrrrr]<-2.0ex>     & -5 & -4 & -3 & -2 & -1 & 0 & \ 1&    & &    \\
1          & \bullet  & &  \bullet & & \bullet &&      \\
2       &  &  \bullet  & &  \bullet  && \bullet &&    \\
3       &    & &  \bullet && \bullet && \bullet   \\
& 
}}
$$	

We now consider the cluster structure of $K_0(\Ctwo)$. In this case we have 
$
\caI = \sigma_2(\g) = \caIe \cup \caIf,
$ where 
\begin{align*}
\caIe &= \{ (1,-1), (1,-3), (2,0), (2,-2), (3,1), (3,-1) \}, \\
\caIf &= \{ (1,-5), (2,-4), (3,-3) \}. 
\end{align*}
Thus the KR modules in the initial seed are given as 
\begin{align*}
\xymatrix@C=0.7pc @R=.2pc{
M_{1, -1} = L(Y_{1,-1}), & M_{1, -3} = L(Y_{1,-3} Y_{1,-1}), &M_{1, -5} = L(Y_{1,-5} Y_{1,-3} Y_{1,-1}), \\
M_{2, 0} = L(Y_{2,0}), & M_{2, -2} = L(Y_{2,-2}Y_{2,0}), & M_{2, -4} = L(Y_{2,-4}Y_{2,-2}Y_{2,0}),\\ 
 M_{3, 1} = L(Y_{3,1}), & M_{3, -1} = L(Y_{3,-1}Y_{3,1}), & M_{3, -3} = L(Y_{3,-3}Y_{3,-1}Y_{3,1}),
}
\end{align*}
and the initial quiver is given as follows:
\begin{align*} \scriptsize
\xymatrix@R=.2pc{ 
 & & \ar[ld] [M_{3,1}]  \\
  & \ar[ld] \ar[rd] [M_{2,0}]&  \\
 \ar[rd]  [M_{1,-1}] & & \ar[ld] \ar[uu] [M_{3,-1}]  \\
 & \ar[ld] \ar[rd] \ar[uu] [M_{2,-2}] &  \\
  \ar[uu] [M_{1,-3}] \ar[rd] & & \ar[uu] \text{\begin{tabular}{|c|} \hline  $[M_{3,-3}]$\\  \hline \end{tabular}}  \\
  & \ar[uu] \text{\begin{tabular}{|c|} \hline  $[M_{2,-4}]$\\  \hline \end{tabular}} & \\
\ar[uu] \text{\begin{tabular}{|c|} \hline  $[M_{1,-5}]$\\  \hline \end{tabular}} & &  \\
}
\end{align*}
Here, the boxed vertices are frozen variables and the arrows between frozen variables are not drawn in the quiver.

\item Let $U_q'(\g)$ be of affine type $B_3^{(1)}$. In this case,  $\gf$ is of type $A_5$, $\If = \{1,2,3,4,5 \}$, and the map $\pi: \If \rightarrow I_0$ is given as follows:
\begin{align*}
	\small
& \xymatrix@R=0.2em@C=1.5em{
 &   1 \ar@{-}[r] \ar@{.}[dd] & 2 \ar@{-}[dr]  \ar@{.}[dd] \\
 A_{5} \ar[ddd]_\pi :  &&& 3  \ar@{.>}[ddd] \\
&    5 \ar@{-}[r] \ar@{.>}[dd] & 4 \ar@{-}[ru] \ar@{.>}[dd]  \\
&    &&&& \\
B_{3} :&   1 \ar@{-}[r] & 2 \ar@{=>}[r] & 3 
}
\end{align*}
Note that $d_\im = 2 - \delta_{\im, 3}$ for $\im\in \If$.
 Define a height function $\xi\cl \If \rightarrow \Z$ by $\xi (1) = -2$, $\xi (2) = 0$, $\xi (3) = 1$, $\xi (4) = 2$ and $\xi (5) = 4$, and set $\ell=2$. We then have 
 \begin{align*}
\sigma^\circ(\g) &= \{  (\im, p_\im) \in \If \times \Z \mid 
p_1, p_4 \in 4 \Z+2,\ p_2, p_5 \in 4 \Z, \ p_3 \in 2\Z +1   \}, \\
\sigma_2(\g) &= \{  (\im, p_\im) \in \sigma_0(\g) \mid  \xi(\im) - 10 \le p \le \xi(\im) \}.
\end{align*}
Here $\sigma_2(\g)$ can be drawn pictorially as follows:
$$ 
\scalebox{0.8}{\xymatrix@C=1ex@R=  0.0ex{ 
\imath \diagdown\, p     \ar@{-}[dddddd]<3ex> \ar@{-}[rrrrrrrrrrrrrrrrr]<-2.0ex>   & & -10 & -9 &  -8 & -7 & -6 & -5 & -4 & -3 & -2 & -1 & 0 & 1 & 2& 3 & 4&  & &    \\
1  && \bullet  & &   &  & \bullet  & &   & & \bullet &&  &&    \\
2  &&  && \bullet  & &  & & \bullet &  &    & &  \bullet  &&  &&  &&   \\
3  &&     &  \bullet & & \bullet & &  \bullet  & &  \bullet && \bullet && \bullet &&  & \\
4  &&    & &  & & \bullet & &    & &  \bullet &&  && \bullet &&  \\
5  &&    & &   & & & &  \bullet  & &   && \bullet &&  && \bullet \\
& 
}}
$$
Consider the cluster structure of $K_0(\Ctwo)$. In this case 
$\caI = \sigma_2(\g) = \caIe \cup \caIf $, where $\caIe= \caI \setminus \caIf $ and 
\begin{align*}
\caIf = \{ (1,-10), (2,-8), (3,-9), (4,-6), (5,-4) \}. 
\end{align*}
Thus the KR modules in the initial seed are given as 
\begin{align*}
\xymatrix@C=0.7pc @R=.2pc{
M_{1, -2} = L(Y_{1,-2}), & M_{1, -6} = L(Y_{1,-6} Y_{1,-2}), &M_{1, -10} = L(Y_{1,-10} Y_{1,-6} Y_{1,-2}), \\
M_{2, 0} = L(Y_{2,0}), & M_{2, -4} = L(Y_{2,-4} Y_{2,0}), &M_{2, -8} = L(Y_{2,-8} Y_{2,-4} Y_{2,0}), \\
M_{3, 1} = L(Y_{3,1}), & M_{3, -1} = L(Y_{3,-1}Y_{3,1}), & M_{3, -3} = L(Y_{3,-3}Y_{3,-1}Y_{3,1}),\\ 
M_{3, -5} = L(Y_{3,-5}\cdots Y_{3,1}), & M_{3, -7} = L(Y_{3,-7}\cdots Y_{3,1}), & M_{3, -9} = L(Y_{3,-9}\cdots Y_{3,1}),\\ 
 M_{4, 2} = L(Y_{2,2}), & M_{4, -2} = L(Y_{2,-2}Y_{2,2}), & M_{4, -6} = L(Y_{2,-6}Y_{2,-2}Y_{2,2}),\\
  M_{5, 4} = L(Y_{1,4}), & M_{5, 0} = L(Y_{1,0}Y_{1,4}), & M_{5, -4} = L(Y_{1,-4}Y_{1,0}Y_{1,4}),
}
\end{align*}
and the initial quiver is given as follows:
\begin{align*} \scriptsize
\xymatrix@R=.0pc{
 & &  & & \ar[ddl] [M_{5,4}]  \\
  & &  & &   \\
  & & & \ar[dddl] \ar[ddr] [M_{4,2}] & \\
  & & \ar[dl] [M_{3,1}] &  & \\
  & \ar[ddl] \ar[dddr] [M_{2,0}] &  &  & \ar[ddl] \ar[uuuu] [M_{5,0}] \\
 &  & \ar[dr] \ar[uu] [M_{3,-1}]  &  &  \\  
\ar[ddr] [M_{1,-2}] &  &  &  \ar[dddl]\ar[ddr] \ar[uuuu] [M_{4,-2}] &  \\ 
 &  & \ar[dl] \ar[uu] [M_{3,-3}] &   &  \\ 
 & \ar[uuuu] \ar[ddl] \ar[dddr] [M_{2,-4}] &  &   & \ar[uuuu] \text{\begin{tabular}{|c|} \hline  $[M_{5,-4}]$\\  \hline \end{tabular}}  \\  
&  & \ar[dr] \ar[uu] [M_{3,-5}] &   &  \\
\ar[uuuu] \ar[ddr] [M_{1,-6}] &  &  & \ar[uuuu] \text{\begin{tabular}{|c|} \hline  $[M_{4,-6}]$\\  \hline \end{tabular}}  &   \\ 
 &  & \ar[dl] \ar[uu] [M_{3,-7}] &   &   \\  
& \ar[uuuu] \text{\begin{tabular}{|c|} \hline  $[M_{2,-8}]$\\  \hline \end{tabular}} &  &   &   \\  
&  & \ar[uu] \text{\begin{tabular}{|c|} \hline  $[M_{3,-9}]$\\  \hline \end{tabular}} &   &   \\  
 \ar[uuuu] \text{\begin{tabular}{|c|} \hline  $[M_{1,-10}]$\\  \hline \end{tabular}} &  &  &   &   \\  
}
\end{align*}
Here, the boxed vertices are frozen variables and the arrows between frozen variables are not drawn in the quiver.

\ee
\end{example}

\section{Verlinde rings and cluster algebras}\label{sec:main_conj}
\subsection{Positivity conjecture}
Let $\gzero$ be a simple Lie algebra, and $\levk \ge 2$ be a positive integer.
From now on, we write $\ellk := \levk - 1$.
We consider the ring homomorphism \(\phik: K_0(\repuqghat) \to K_0(\fusk)\), defined as the composition of the homomorphisms 
\eqref{eq:res_hom}, \eqref{eq:iso_hom}, and \eqref{eq:Valcove}:
\begin{equation}\label{eq:phik}
\phik : K_0(\repuqghat) \xrightarrow{\res} K_0(\repuqgzero) \xrightarrow{\cong} K_0(\repg) \xrightarrow{\pik} K_0(\fusk).
\end{equation}
Note that one can define the $\RR$-valued map $\qdimk$ on each of the rings $K_0(\repuqghat)$, $K_0(\repuqgzero)$, and $K_0(\repg)$
by pulling back $\qdimk$ from $K_0(\fusk)$.

In this subsection, we present our main conjecture regarding the positivity in the Verlinde ring $K_0(\fusk)$ under the ring homomorphism $\phik$ on the Grothendieck ring of the monoidal subcategory $\Cellk$ of $\repuqghat$.

In \cite{MR3628223}, we studied the image of the KR modules under this map $\phik$.
Let us briefly review a conjecture formulated there.
For each $(i,m,p) \in I_0\times \Znat\times \Z$, let $\Wimp$ be the corresponding KR module.
By restriction, we get a finite-dimensional $\uqg$-module $\res \Wimp$, whose isomorphism class is independent of $p$, and
for simplicity, we use the notation $\phik([\Wim])$.

For \(i \in I_0\), let \(t_i := d/d_i \in \{1, d\}\), where \(d_i\) is given in \eqref{eq:d_i}, and \(d = \max_{i \in I_0} \{ d_i \} \in \{1, 2, 3\}\).  
Alternatively, as noted in the Introduction, \(t_i\) is the ratio of the norm of the highest root to that of the simple root \(\alpha_i\).
Thus, \(t_i = 2/(\alpha_i \mid \alpha_i)_0\), with the bilinear form \((\cdot \mid \cdot)_0\) normalized such that \((\theta \mid \theta)_0 = 2\).
We have the following conjecture:
\begin{conjecture} \cite[Conjecture 3.10]{MR3628223} \label{conj:phik}
For $i\in I_0$, the following properties hold :
\begin{enumerate}
\item $\phik([\Wim])$ is positive if $0\le m \le t_i \levk$.
\item $\phik([\Wim])$ is a unit if $m=0$ or $m = t_i \levk$.
\item $\phik([\Wim])=0$ if $m = t_i \levk+1$.
\end{enumerate}
\end{conjecture}
For a more precise version of the conjecture, including a detailed description of the units and the values of $\phik([\Wim])$ for $m > t_i \levk + 1$, see \cite{MR3628223}.
Especially, the conjecture above implies the following conjecture, originally proposed by Kirillov \cite{MR947332} and later expanded by Kuniba, Nakanishi and Suzuki \cite[Conjecture 14.2]{MR2773889}.
\begin{conjecture}[Theorem in some cases stated below]\label{conj:qconj}
For $i\in I_0$, the following properties hold :
\begin{enumerate}
\item $\qdimk \Wim$ is positive if $0\le m \le t_i \levk$.
\item $\qdimk \Wim=1$ if $m=0$ or $m = t_i \levk$.
\item $\qdimk \Wim=0$ if $m = t_i \levk+1$.
\end{enumerate}
\end{conjecture}
This quantum dimension version of the conjecture is known to hold for classical types \cite{MR3628223}, as well as for type $E_6$ (for all $i \in I_0$) and types $E_7$ and $E_8$ (for certain nodes depending on $i \in I_0$) \cite{MR3282650}.
See also \cite{MR4342381} for a discussion on the case of $\levk=1$ for the exceptional types $E_6$, $E_7$, and $E_8$ and its applications.

\begin{prop}
Assume that \conjref{conj:phik} or \conjref{conj:qconj} holds for all $i\in I_0$.
Then $\qdimk \Wim>1$ for $i\in I_0$ and $1 \le m \le t_i \levk-1$.
\end{prop}
\begin{proof}
For $0 \le m \le t_i \levk$,
let $\Diim=\qdimk \Wim$ be the quantum dimension of $\Wim$ at level $\levk$, which must be positive under the given assumption.
Then the $Q$-system \cite{MR2254805}
\begin{equation}\label{eq:Qsys}
(\Qim)^2 - Q^{(i)}_{m+1}Q^{(i)}_{m-1} = \prod_{j : a_{i,j}< 0} \prod_{n=0}^{-a_{i,j}-1}
Q^{(j)}_{\bigl \lfloor\frac{a_{j,i}m - n}{a_{i,j}}\bigr \rfloor},\quad i\in I_0, m \ge 1
\end{equation}
satisfied by the characters of KR modules, implies that the finite sequence $(\Diim)_{0\le m \le t_i \levk}$ of real positive numbers is log-concave, meaning that $(\Diim)^2\ge \Dii_{m+1}\Dii_{m-1}$.
Thus, it must be unimodal.
As $\Dii_{0}=\Dii_{t_i \levk}=1$, there are two possibilities for the unimodal behavior:
either $\Diim=1$ for all $0\le m \le t_i \levk$ or $\Diim>1$ for all $1\le m \le t_i \levk-1$.
However, the first case cannot hold due to \eqref{eq:Qsys}.
Therefore, $\Diim>1$ for all $1\le m \le t_i \levk-1$, which gives the desired conclusion.
\end{proof}

While \conjref{conj:phik} raises an interesting positivity problem concerning $K_0(\repuqghat)$ and $K_0(\fusk)$, it also poses the challenge of identifying an appropriate domain within $K_0(\repuqghat)$ where the map $\phik$ consistently produces a positive image.
Notably, for a KR module $V$, $\phik([V])$ can be zero or even negative.

Based on the conjectural positivity of $\phik([\Wim])$ for $i \in I$ and $0 \le m \le t_i \levk$, along with the construction of $\Cellk$ as a monoidal subcategory of $\repuqghat$, where the same KR modules in $\Cellk$ appear as cluster variables in the initial seed of the cluster algebra structure on $K_0(\Cellk)$, we propose the following:
\begin{conjecture}\label{conj:main1}
Let $\levk \ge 2$ be a positive integer, and let $\phik : K_0(\Cellk) \to K_0(\fusk)$ be the restriction of $\phik$ of \eqref{eq:phik} on $K_0(\Cellk)$ which is a subring of $K_0(\repuqghat)$.
\begin{enumerate}
\item For every $V$ in $\Cellk$, $\phik([V])$ is positive in $K_0(\fusk)$.
\item For $V$ in $\Cellk$, $\phik([V])$ is a unit if and only if $V$ is a frozen module.
\end{enumerate}
\end{conjecture}

We also propose a weaker version of the above conjecture in terms of the quantum dimension.
\begin{conjecture}\label{mainconj2}
Let $\levk \ge 2$ be a positive integer, and let $\qdimk : K_0(\Cellk) \to \RR$ be the quantum dimension at level $\levk$.
\begin{enumerate}
\item For every $V$ in $\Cellk$, $\qdimk V>0$.
\item For $V$ in $\Cellk$, $\qdimk V=1$ if and only if $V$ is a frozen module.
\end{enumerate}
\end{conjecture}

The remainder of the paper is focused on proving the conjecture in some special cases and presenting supporting evidences.

\subsection{Positivity for quantum dimensions}
\begin{theorem}\label{thm:qdim}
Assume that \conjref{conj:qconj} holds.
Then every $V$ in $\Cellk$ has positive quantum dimension.
\end{theorem}
\begin{proof}
Let $\ell = \ellk$.
Recall the initial seed $\seed = ( \{[M_{\im, p} ]\}_{(\im, p) \in \caI},  \caB)$ of $K_0(\Cell)$ given in \eqref{Eq: initial seed}.
Let $<$ be a total order on $\caI$, and define 
\[ 
M(\bfa):= \oprod_{x \in \caI} M_{x}^{\tens a_{x}} \qquad \text{ for any $\bfa = (a_x)_{x\in \caI} \in \Z_{\ge0}^{\oplus \caI}$,}
\]
where we write
$
\oprod_{j \in J} A_j = \cdots \tens A_{j_{-2}} \tens A_{j_{-1}} \tens A_{j_0} \tens A_{j_{1}}  \tens \cdots 
$
for a totally ordered set $J = \{ \cdots < j_{-1} < j_0 < j_1 < j_2 < \cdots \}$. 
Note that the element $[M(\bfa)]$ in $K_0(\Cell)$ does not depend on the choice of total orders $<$ on $\caI$.

We now consider the Laurent phenomenon of the cluster algebra $K_0(\Cell)$ with respect to the seed $\seed$, i.e., 
\[
\psi: K_0(\Cell)  \hookrightarrow \Z[ [M_x]^{\pm}, [M_y] \mid x \in \caIe, y\in \caIf].
\]
Let $L$ be a simple module in $\Cell$. 
Then there exists $\bfa \in \Znat^{\oplus \caI}$ such that 
\begin{align} \label{Eq: L M(a)}
[L \tens M(\bfa)] = \sum_k c_k [M(\bfb_k)] \qquad \text{for some $c_k \in \Znat$ and $\bfb_k \in\Znat^{\oplus \caI}$,}
\end{align}
where $k$ runs over some finite index set. 
Note that, since $L \tens M(\bfa)$ is a module and all $M(\bfb_k)$ are simple in $\Cell$, the coefficients $c_k$ should be non-negative integers. 
Applying $\qdimk$ to the equation \eqref{Eq: L M(a)}, we conclude that 
$\qdimk L > 0$ by the assumption.
\end{proof}

\section{Verlinde positivity for cluster algebras of finite cluster type}\label{sec:pos_fit_cls}

In this section, we take the following height function of $\Cell$ of type $ADE$: 
\begin{equation}\label{eq:height-function}
\xi(i) = 
\begin{cases}
2-i & \text{for type } A_n, \\
i-2 & \text{for type } D_n, \, i \in [n-1], \\
n-3 & \text{for type } D_n, \, i=n, \\
-1 & \text{for type } E_n, \, i=1, \\
2 & \text{for type } E_n, \, i=2, \\
0 & \text{for type } E_n, \, i=3, \\
i-3 & \text{for type } E_n, \, i\in [4,n].
\end{cases}
\end{equation}
Note that for type $E_n$, $n$ belongs to the set $\{6,7,8\}$, and the diagrams for $E_7$ and $E_6$ can be obtained by successively removing vertices labeled $8$ and $7$, respectively, from the $E_8$ diagram:
\begin{center}
\begin{tikzpicture}[scale=0.5]
\draw (0 cm,0) -- (12 cm,0);
\draw (4 cm, 0 cm) -- +(0,2 cm);
\draw[fill=white] (0 cm, 0 cm) circle (.25cm) node[below=4pt]{$1$};
\draw[fill=white] (2 cm, 0 cm) circle (.25cm) node[below=4pt]{$3$};
\draw[fill=white] (4 cm, 0 cm) circle (.25cm) node[below=4pt]{$4$};
\draw[fill=white] (6 cm, 0 cm) circle (.25cm) node[below=4pt]{$5$};
\draw[fill=white] (8 cm, 0 cm) circle (.25cm) node[below=4pt]{$6$};
\draw[fill=white] (10 cm, 0 cm) circle (.25cm) node[below=4pt]{$7$};
\draw[fill=white] (12 cm, 0 cm) circle (.25cm) node[below=4pt]{$8$};
\draw[fill=white] (4 cm, 2 cm) circle (.25cm) node[right=3pt]{$2$};
\end{tikzpicture}
\end{center}

The aim of this section is to establish the following result:
\begin{theorem}\label{thm:Cone}
Let $\gzero$ be of type $A_n$, $E_6$, $E_7$, or $E_8$ with level $\levk = 2$, or let $\gzero$ be of type $A_1$ with $\levk \ge 2$. We choose the height functions for $\Cellk$ as above. 
Let $\phik : K_0(\Cellk) \to K_0(\fusk)$ be the homomorphism from \conjref{conj:main1}. 
For every $V$ in $\Cellk$, $\phik([V])$ is positive in $K_0(\fusk)$.
\end{theorem}

We can reduce the problem to checking the positivity for the cluster variables, using the following result:

\begin{prop} \label{Prop: FC type}
Let $R$ be an algebra over $\Z$ and let $B$ be a linear basis of $R$.
Let $\phi: K_0(\Cell) \rightarrow R$ be a ring homomorphism.
Suppose that 
\bna
\item $K_0(\Cell)$ is of finite cluster type, 
\item for any $a,b\in B$, the product $ab$ is non-zero and can be written as a linear combination of $B$ with non-negative integers,
\item for any cluster variable $x$ in $K_0(\Cell)$, $\phi(x)$ is non-zero and can be written as a linear combination of $B$ with non-negative integers.
\ee
Then, for any simple module $L \in \Cell$, the element $\phi( [L] )$ is non-zero and can be written as a linear combination of $B$ with non-negative integers.
\end{prop}
\begin{proof}
Let $G$ be the set of isomorphism classes of all simple modules in $ \Cell$, and let $C$ be the set of all cluster monomials of $K_0(\Cell)$. Note that $G$ is a linear basis of $K_0(\Cell)$. It was shown in \cite{KKOP24} that any cluster monomial corresponds to a simple module via the categorification between $K_0(\Cell)$ and $\Cell$, establishing that $C \subset G$.

On the other hand, by the assumption (a), \thmref{Thm: cm basis} says that the cluster monomials form a linear basis of $K_0(\Cell)$. We thus conclude that $G=C$.
Therefore, the assertion follows from the assumptions (b) and (c). 
\end{proof}

\subsection{Type \texorpdfstring{$A_n$}{An}} 
We choose the height function $\xi(i)=2-i$ for type $A_n$ as in \eqref{eq:height-function}.
\begin{prop}
For any $\ell \ge 1$, all simple objects in $\Cell^{A_1}$ and $\Cone^{A_{\ell}}$ are sent to positive elements in the Verlinde ring. 
\end{prop}

\begin{proof}
Recall $\levk = \ell+1$.
By Section 13 in \cite{MR2682185}, $K_0(\Cell^{A_1}) \cong \C[\Gr(2,\ell+3, \sim)] \cong \C[\Gr(\ell+1,\ell+3, \sim)] \cong K_0(\catC_{1}^{A_\ell})$. By \cite{MR1483901}, all prime modules of $U_q(\widehat{\sl_2})$ are KR modules. Since $\Cell^{A_1}$ is of finite cluster type, the set of prime modules in $\Cell^{A_1}$ are the same as cluster variables in $\Cell^{A_1}$. Therefore all cluster variables in $\Cell^{A_1}$ are KR modules. These cluster variables are $L(Y_{1,-1}Y_{1,-3}\cdots Y_{1,-2j+1})$, $j \in [1, \ell+1]$. 
As a $U_q(\sl_2)$-module, it is isomorphic to $L(j \poid_1)$ and is mapped to 
$[L((\levk-j) \La_0+ j \La_1)]$ with 
$(\levk-j) \La_0+ j \La_1 \in \Pkp$ as given in \eqref{eq:Valcove}.
Therefore, all simple objects in $\Cell^{A_1}$ are sent to positive elements in the Verlinde ring.

The isomorphism $\C[\Gr(2,\ell+3, \sim)] \to \C[\Gr(\ell+1,\ell+3, \sim)]$ sends a Pl\"{u}cker coordinate $P_{i_1,i_2}$ to $P_{[\ell+3] \setminus \{i_1,i_2\}}$, $1 \le i_1 < i_2+1 \le \ell+3$. Since cluster variables in $\C[\Gr(2,\ell+3, \sim)]$ are $P_{i_1,i_2}$, $1 \le i_1 < i_2+1 \le \ell+3$, we have that the cluster variables in $\C[\Gr(\ell+1,\ell+3, \sim)]$ are $P_{[\ell+3] \setminus \{i_1,i_2\}}$, $1 \le i_1 < i_2+1 \le \ell+3$. By Theorem 3.17 in \cite{MR4159839}, the Pl\"{u}cker coordinates $P_{[\ell+3] \setminus \{i_1,i_2\}}$ correspond to evaluation modules $L(Y_{\ell+2-i_1, \ell-i_1}Y_{\ell+3-i_2,\ell-i_2-1})$.
We set $\poid_{0}=0$, and $Y_{0,s}=1$ for $s \in \Z$. 
As a $U_q(\sl_{\ell+1})$-module, it is isomorphic to $L(\poid_{\ell+2-i_1}+\poid_{\ell+3-i_2})$ and is mapped to 
$[L(\La_{\ell+2-i_1}+\La_{\ell+3-i_2})]$ as in \eqref{eq:Valcove}. It follows that all cluster variables in $K_0(\catC_{1}^{A_\ell})$ are sent to positive elements in the Verlinde ring.
\end{proof}

\subsection{Type \texorpdfstring{$D_n$}{Dn}} 
Let $\levk=2$ and $\ell=\ell(\levk)=1$. The number of cluster variables of type $D_n$ is $n^2$. We take the height function as in \eqref{eq:height-function}. 

Label the mutable vertex of the initial quiver with module $L(Y_{i,i-2})$ as $i$, $i \in [n-1]$, and label the mutable vertex with module $Y_{n,n-3}$ as $n$. Starting from the initial seed, using the mutation sequence 
\begin{align*}
1,2,\ldots,n,1,2,\ldots,n,\ldots, 1,2,\ldots,n, \quad (\text{$n$ copies of $1,2,\ldots,n$})
\end{align*}
the seed becomes the initial seed again.
We list the dominant monomials corresponding to cluster variables in $\Cone$ of type $D_n$ in \tableref{tab:dominant monomials of cluster variables in C1 of type Dn}, where we divide the cluster variables into four groups according to the degree of dominant monomials corresponding to the cluster variables. The number of cluster variables in Table \ref{tab:dominant monomials of cluster variables in C1 of type Dn} is $n^2$ which is equal to the number of cluster variables of type $D_n$.

\begin{table}
\centering
\begin{tabular}{|c|c|c|}
\hline
Degree & Dominant monomials & Number \\
\hline
$1$ & $Y_{i,i-4}$, $Y_{i,i-2}$, $i \in [n-1]$, $Y_{n,n-3}$, $Y_{n,n-5}$ & $2n$ \\
\hline
$2$ & $Y_{i,i-4}Y_{j,j-2}$, $1 \le i < j \le n-1$, $Y_{i,i-4}Y_{n,n-3}$, $i \in [n-2]$ & $\frac{(n-2)(n+1)}{2}$ \\
\hline
$3$ & $Y_{i,i-4}Y_{n-1,n-3}Y_{n,n-3}$, $i \in [n-2]$ & $n-2$ \\
\hline
$4$ & $Y_{i,i-4}Y_{j,j-4}Y_{n-1,n-3}Y_{n,n-3}$, $1 \le i < j \le n-2$ & $\frac{(n-2)(n-3)}{2}$ \\
\hline
\end{tabular}
\caption{Dominant monomials of cluster variables in $\Cone$ of type $D_n$.}
\label{tab:dominant monomials of cluster variables in C1 of type Dn}
\end{table}

We also compute the images of cluster variables in $\mathcal{C}_{1}$ of type $D_n$ in the Verlinde ring. We have the following conjecture (we checked the conjecture for $4 \le n \le 8$). 
\begin{conjecture}
The cluster variables in Table \ref{tab:dominant monomials of cluster variables in C1 of type Dn} are all cluster variables in $\mathcal{C}_{1}^{D_n}$. The images of cluster variables in $\mathcal{C}_{1}^{D_n}$ in the Verlinde ring are as follows. 
\begin{itemize}
\item  For $i \in [n-2]$, 
\begin{align*}
\phik([L(Y_{i,i-2}]))=\phik([L(Y_{i,i-4})]) & = \begin{cases} 
[L(\La_0+\La_1)]+ \sum_{s=1}^{\frac{i-1}{2}} [L(\La_{2s+1})], & \text{$i$ is odd}, \\
[L(2\La_0)]+ \sum_{s=1}^{\frac{i}{2}} [L(\La_{2s})], & \text{$i$ is even},
\end{cases} \\
\phik([L(Y_{i,i-2}Y_{i,i-4})]) & = \begin{cases}
[L(2\La_1)], & \text{$i$ is odd}, \\
[L(2\La_0)], & \text{$i$ is even}.
\end{cases}
\end{align*}

\item For $i \in \{n-1,n\}$, $\phik([L(Y_{i, n-3})])$ and $\phik([L( Y_{i, n-5} )])$ are equal to $L(\La_0+\La_i)$, and $\phik([L(Y_{i,n-5} Y_{i,n-3})]) = L(2\La_{i})$. 

\item For $i \in [n-2]$, $j \in \{n-1,n\}$,
\begin{align*} 
& \phik([L( Y_{i,i-4} Y_{j,n-3} )]) = \begin{cases}
[L(\La_1+\La_j)], & \text{$i$ is odd}, \\
[L(\La_0+\La_j)], & \text{$i$ is even}.
\end{cases} 
\end{align*}

\item For $1 \le i < j \le n-2$,
\begin{align*}
\phik([L(Y_{i,i-4}Y_{j,j-2})]) = \begin{cases}
[L(\La_0+\La_1)] + \sum_{s=0}^{\frac{j-i-3}{2}} [L(\La_{2s+3})], & \text{$j-i$ is odd}, \\
[L(2\La_1)] + \sum_{s=0}^{\frac{j-i-2}{2}} [L(\La_{2s+2})], & \text{$j-i$ is even}.
\end{cases}
\end{align*} 

\item For $i \in [n-2]$, $\phik([L(Y_{i,i-4}Y_{n-1,n-3}Y_{n,n-3})])$ is equal to
\begin{align*}
\begin{cases}
[L(\La_0+\La_1)]+ \sum_{s=0}^{\frac{n-i-4}{2}} [L(\La_{2s+3})], & \text{$n-i$ is even}, \\
[L(2\La_0)]+ \sum_{s=0}^{\frac{n-i-3}{2}} [L(\La_{2s+2})], & \text{$n-i$ is odd, $n$ is odd}, \\
[L(2\La_1)]+ \sum_{s=0}^{\frac{n-i-3}{2}} [L(\La_{2s+2})], & \text{$n-i$ is odd, $n$ is even}.
\end{cases}
\end{align*} 

\item For $1 \le i < j \le n-2$, $\phik([L(Y_{i,i-4}Y_{j,j-4}Y_{n-1,n-3}Y_{n,n-3})])$ is equal to
\begin{align*}
\begin{cases} 
[L(\La_0+\La_1)]+ \sum_{s=0}^{\frac{n-j+i-4}{2}} [L(\La_{2s+3})], & \text{$n-j+i$ is even}, \\
[L(2\La_0)]+ \sum_{s=0}^{\frac{n-j+i-3}{2}} [L(\La_{2s+2})], & \text{$n-j+i$ is odd, $n$ is odd}, \\
[L(2\La_1)]+ \sum_{s=0}^{\frac{n-j+i-3}{2}} [L(\La_{2s+2})], & \text{$n-j+i$ is odd, $n$ is even}.
\end{cases}
\end{align*} 
\end{itemize}
 
\end{conjecture}

\subsection{Exceptional types $E_6$, $E_7$, and $E_8$}
In this subsection, we consider the exceptional types. Let $\gzero$ be of type $E_6$, $E_7$, or $E_8$, with $\levk = 2$. 
We choose the height function as in \eqref{eq:height-function}.
Since the ring $K_0(\Cellk)$ has a cluster algebra structure of finite cluster type, it contains only finitely many cluster variables and finitely many exchange relations. We begin by briefly outlining the strategy for proving \thmref{thm:Cone}, noting that this approach can also be applied to other types, as long as that they possess a cluster algebra structure of finite cluster type.

Let $\{x_j : j \in J\}$ denote the set of cluster variables of $K_0(\Cellk)$, where $J$ is a finite set. Then,
\[
\phik(x_j) = \sum_{\la \in \Pkp} c(j, \la) [L(\la)] \in \fusk
\]
for some integers $c(j, \la) \in \Z$. Our goal is to show that $c(j, \la)$ is non-negative for every pair $(j, \la) \in J \times \Pkp$ and that for a fixed $j$, not all $c(j, \la)$ are zero.

If a cluster variable is attached to a KR module, we can explicitly determine its image in the Verlinde ring under the map $\phik$. Specifically, the decomposition of a KR module into a direct sum of irreducible $\uqgzero$-modules can be obtained using the Fermionic formula \cite{MR1745263}.
Moreover, the image under the homomorphism $\pik : K_0(\repg) \to K_0(\fusk)$ can be algorithmically determined by employing the affine Weyl group as explained in \subsecref{subsec:verlinde}.

In general, the main challenge in computing the image under $\phik$ lies in the absence of an efficient method for decomposing a $\uqghatp$-module into a direct sum of irreducible $\uqg$-modules when it is not a KR module. In some cases, we may compute the $q$-character, truncated $q$-character, or classical character as a $\uqg$-module
using the Frenkel-Mukhin algorithm \cite{MR1810773}, or by exploiting the cluster algebra structure as in \cite{MR3500832}.
However, these methods are often impractical, particularly because the modules involved, especially in exceptional types, tend to have very large dimensions. This leads to considerable computational complexity, making character calculations and module decompositions inefficient and difficult to execute in practice.

Instead, we can impose the exchange relations as constraints on the coefficients $\{c(j, \la)\}$. Since $\phik$ is a ring homomorphism, applying it to the exchange relations yields a finite system of polynomial equations involving $\{c(j, \la)\}$.
The coefficients associated with a KR module can be determined directly. From there, it is sufficient to verify that this partial integral solution can be uniquely extended to an integral solution for the entire system of equations. Afterward, we can confirm whether all $c(j, \la)$ are non-negative.

The exchange relations generally lead to a large system of non-linear Diophantine equations for \(c(j, \la)\). Instead of solving the entire system at once, we iteratively solve only the linear equations that arise after substituting the known values of \(c(j, \la)\) into all exchange relations at each step.
Remarkably, only two iterations are needed for all types \(E_6\), \(E_7\), and \(E_8\).
A Mathematica implementation of this approach, together with all cluster variables and exchange relations, is accessible at \url{https://github.com/chlee-0/verlinde_cluster}.
The following subsections will provide further details for each type. Recall that a cluster variable in $\mathcal{C}_{\ell}$ is determined by a dominant monomial in $Y_{i,s}$. We will describe the cluster variables in terms of dominant monomials.

\subsubsection{Type \texorpdfstring{$E_6$}{E6}}
Let $\gzero$ be of type $E_6$, with $\levk = 2$.
In this case, the Verlinde ring \(K_0(\fusk)\) consists of 9 basis elements, with \(V_0\) acting as the identity element. The basis elements are:  
\[
\begin{aligned}
V_0 &= [L(2\La_0)], & V_1 &= [L(\La_0 + \La_1)], & V_2 &= [L(\La_0 + \La_6)], \\
V_3 &= [L(2\La_1)], & V_4 &= [L(\La_1 + \La_6)], & V_5 &= [L(\La_2)], \\
V_6 &= [L(\La_3)], & V_7 &= [L(\La_5)], & V_8 &= [L(2\La_6)].
\end{aligned}
\]
The products among \(V_i\) can be determined using the surjective homomorphism \(\pik\) given in \eqref{eq:Valcove}.
There are 48 cluster variables in total, which include 42 exchangeable variables and 6 frozen ones.
See Table \ref{tab:E6_cluster} in Appendix \ref{app:A} for the list of cluster variables along with their corresponding dominant monomials.
Among these variables, there are 18 KR modules, specifically: $x_1,x_2, \dots, x_{12}$, $x_{26}$, $x_{27}$, $x_{28}$, $x_{30}$, $x_{31}$, and $x_{46}$.
For the KR modules, we have the following:
\begin{small}
\[
\begin{aligned}
    \phik(x_1) &= V_1, \quad &\phik(x_2) &= V_3, \quad &\phik(x_3) &= V_2 + V_6, \quad &\phik(x_4) &= V_8, \\
    \phik(x_5) &= V_0 + V_4 + 2V_5, \quad &\phik(x_6) &= V_0, \quad &\phik(x_7) &= V_0 + V_5, \quad &\phik(x_8) &= V_0, \\
    \phik(x_9) &= V_1 + V_7, \quad &\phik(x_{10}) &= V_3, \quad &\phik(x_{11}) &= V_2, \quad &\phik(x_{12}) &= V_8, \\
    \phik(x_{26}) &= V_1, \quad &\phik(x_{27}) &= V_2, \quad &\phik(x_{28}) &= V_2 + V_6, \quad &\phik(x_{30}) &= V_0 + V_5, \\
    \phik(x_{31}) &= V_0 + V_4 + 2V_5, \quad &\phik(x_{46}) &= V_1 + V_7.
\end{aligned}
\]
\end{small}

Let us write
\[
\phik(x_j)  = \sum_{k=0}^{8}c(j, k)V_k, \quad c(j,k)\in \Z
\]
for each $j=1,\dots, 48$.
At this point, there are \(18 \times 9 = 162\) known values of \(c(j, k)\) from the KR modules, leaving \((48 - 18) \times 9 = 270\) values still to be determined.

The complete list of all exchange relations is available at \url{https://github.com/chlee-0/verlinde_cluster/exchange_relations_E6.txt}.
One example of such a relation is:
\[
x_1 x_{34} = x_2 x_3 + x_4.
\]
By applying $\phik$ to the above with the known values of $c(j,k)$ for $j=1,2,3,4$, we obtain the following relations:
\[
\begin{array}{ll}
    c(34, 2) = 0, & -1 + c(34, 0) + c(34, 4) + c(34, 5) = 0, \\
    c(34, 1) + c(34, 7) + c(34, 8) = 0, & c(34, 1) = 0, \\
    c(34, 2) + c(34, 3) + c(34, 6) = 0, & c(34, 2) + c(34, 6) = 0, \\
    c(34, 1) + c(34, 7) = 0, & -1 + c(34, 4) + c(34, 5) = 0, \\
    -1 + c(34, 4) = 0.
\end{array}
\]
In this case, we can determine the values of \(c(34, k)\) for \(k = 0, \dots, 8\) from a single exchange relation. However, this is not always possible. For example, the exchange relation
\[
x_1 x_{17} = x_2 x_{11} + x_{16}
\]
results in the following system of equations:
\begin{small}
\[
\begin{array}{ll}
    -c(16, 0) + c(17, 2) = 0, & -1 - c(16, 1) + c(17, 0) + c(17, 4) + c(17, 5) = 0, \\
    -c(16, 2) + c(17, 1) + c(17, 7) + c(17, 8) = 0, & -c(16, 3) + c(17, 1) = 0, \\
    -c(16, 4) + c(17, 2) + c(17, 3) + c(17, 6) = 0, & -c(16, 5) + c(17, 2) + c(17, 6) = 0, \\
    -c(16, 6) + c(17, 1) + c(17, 7) = 0, & -c(16, 7) + c(17, 4) + c(17, 5) = 0, \\
    -c(16, 8) + c(17, 4) = 0.
\end{array}
\]
\end{small}
This system alone does not determine all the unknowns involved, which makes it necessary to consider multiple exchange relations simultaneously to find the solutions.
As outlined earlier, we iteratively solve the linear equations for \(c(j, k)\) that emerge after substituting the known values of \(c(j, k)\) at each step.

In the first iteration, the known values of \(c(j, k)\) come from the KR modules. Solving the linear equations at this point determines an additional 261 values of \(c(j, k)\) out of the 270 unknowns, leaving $c(48, k)$ for $k=0,\dots, 8$ as the only remaining unknowns.

In the second iteration, solving the linear equations uniquely determines these remaining unknowns. The solution shows that every \(c(j, k)\in \Z\) is non-negative, and moreover,
$\phik(x_j)$ is positive for all $j=1,\dots, 48$. 

For the specific values, refer to Table \ref{tab:E6_cluster} in Appendix \ref{app:A}.
A Mathematica notebook implementing this procedure can be found at \url{https://github.com/chlee-0/verlinde_cluster/blob/main/Verlinde_E6_proof.nb}.

\subsubsection{Type \texorpdfstring{$E_7$}{E7}}
Let $\gzero$ be of type $E_7$ with $\levk = 2$.
The Verlinde ring \(K_0(\fusk)\) consists of 6 basis elements, \(V_0, V_1, \dots, V_5\), with \(V_0\) serving as the identity element:  
\[
\begin{aligned}
V_0 &= [L(2\La_0)], & V_1 &= [L(\La_0 + \La_7)], & V_2 &= [L(\La_1)], \\
V_3 &= [L(\La_2)], & V_4 &= [L(\La_6)], & V_5 &= [L(2\La_7)].
\end{aligned}
\]
There are 77 cluster variables in total, consisting of 70 exchangeable variables and 7 frozen ones.
For a full list of cluster variables $x_1,\dots, x_{77}$ and their corresponding dominant monomials, see Table \ref{tab:E7_cluster} in Appendix \ref{app:A}.
Among these variables, there are 21 KR modules, specifically: $x_1,x_2, \dots, x_{14}$, $x_{31}$, $x_{41}$, $x_{43}$, $x_{44}$, $x_{45}$, $x_{53}$, and $x_{59}$.

Expressing
\[
\phik(x_j)  = \sum_{k=0}^{5}c(j, k)V_k, \quad c(j,k)\in \Z
\]
for each $j=1,\dots, 77$, 
we initially have \(21 \times 6 = 126\) known values of \(c(j, k)\) from the KR modules, with \((77 - 21) \times 6 = 336\) values remaining to be determined.

By iteratively solving the linear equations obtained from the exchange relations, as described earlier, we verify that $\phik(x_j)$ is positive for all $j=1,\dots, 77$. In the first iteration, 288 additional values are determined and in the second iteration, all the remaining values are found.
The specific value can be found in Table \ref{tab:E7_cluster} in Appendix \ref{app:A}.
For a complete list of exchange relations and the Mathematica notebook for solving the linear equations, refer again to \url{https://github.com/chlee-0/verlinde_cluster}.

\subsubsection{Type \texorpdfstring{$E_8$}{E8}}
Let $\gzero$ be of type $E_8$ with $\levk = 2$.
The Verlinde ring \(K_0(\fusk)\) consists of 3 basis elements: \(V_0\), \(V_1\), and \(V_2\), where \(V_0\) serves as the identity element:  
\[
\begin{aligned}
V_0 &= [L(2\La_0)], & V_1 &= [L(\La_1)], & V_2 &= [L(\La_8)].
\end{aligned}
\]
There are 136 cluster variables in total, consisting of 128 exchangeable variables and 8 frozen ones.
For a complete list of cluster variables $x_1,\dots, x_{136}$ and their corresponding dominant monomials, refer to Table \ref{tab:E8_cluster} in Appendix \ref{app:A} 
Among these variables, 24 correspond to KR modules, specifically: $x_1,x_2, \dots, x_{19}$, $x_{34}$, $x_{37}$, $x_{44}$, $x_{113}$, and $x_{114}$.

If we write
\[
\phik(x_j)  = \sum_{k=0}^{2}c(j, k)V_k, \quad c(j,k)\in \Z
\]
for each $j=1,\dots, 136$, 
then initially, there are \(24 \times 3 = 72\) known values of \(c(j, k)\) from the KR modules, with \((136 - 24) \times 3 = 336\) values left to be determined.

By iteratively solving the linear equations obtained from the exchange relations, as described earlier, we verify that $\phik(x_j)$ is positive for all $j=1,\dots, 136$. In the first iteration, 216 additional values are determined and in the second iteration, all the remaining values are found.

The specific value can be found in Table \ref{tab:E7_cluster} in Appendix \ref{app:A}.
For a complete list of exchange relations and the Mathematica notebook for solving the linear equations, refer again to \url{https://github.com/chlee-0/verlinde_cluster}.

\section{Further examples for the positivity conjecture}\label{sec:examples}
In \secref{sec:pos_fit_cls}, we established \conjref{conj:main1} in some cases, and in those cases \( K_0(\Cell) \) has a cluster algebra structure of finite cluster type. 
In this section, we explore examples beyond these cases, demonstrating that the conjecture is still applicable.
A simple object $L(m)$ in $\Cell$ is called real if $L(m) \otimes L(m)$ is simple. Note that every cluster monomial is real. We present an example of the image of a non-real simple module in the Verlinde ring.
We also examine a case where \(\gzero\) is not simply-laced.

\subsection{An example of a non-real module in type $A_2$}
Let \(\gzero\) be of type \(A_2\), with level \(\levk = 6\) and \(\ell = \levk - 1 = 5\).
Consider a simple $U_q(\widehat{\sl_3})$-module $L(m)$ with dominant monomial
\begin{align*}
m = Y_{1,-1}Y_{1,-3}Y_{1,-9}Y_{1,-11}Y_{2,-4}Y_{2,-6}^2Y_{2,-8}.
\end{align*}
Here we choose the height function $\xi(1)=-1$, $\xi(2)=0$. It is easy to check that the indices of $Y_{i,s}$ in $m$ satisfy the condition in (\ref{eq:sigma ell for C ell}). 
Therefore, $L(m)$ is in $\Cell$.
By the correspondence of semistandard Young tableaux and dominant monomials in \cite{MR4159839}[Section 3], the module $L(m)$ corresponds to the semistandard Young tableau $T = \ytableausetup{centertableaux} \scalemath{0.6}{ \begin{ytableau}
1 & 3 & 4 \\
2 & 6 & 7 \\
5 & 8 & 9
\end{ytableau}}$. In Section 8 of \cite{MR4159839}, it is computed that
\begin{align} \label{eq:chT for 125368479}
\ch(T) = P_{125}P_{378}P_{469}-P_{123}P_{459}P_{678}-P_{124}P_{378}P_{569}-P_{126}P_{345}P_{789}.
\end{align}
It was verified that the module $L(m)$ is non-real in Section 8 in \cite{MR4159839}.
By \eqref{eq:chT for 125368479}, we have that
\begin{align*}
& \chi_q(L(m)) = \chi_q(L( Y_{1,-3}Y_{1,-1} )) \chi_q(L( Y_{2,-8}Y_{2,-6}Y_{2,-4} )) \chi_q( L( Y_{1,-11}Y_{1,-9}Y_{2,-6} ) )\\
& - \chi_q(L(Y_{1,-11}Y_{1,-9}Y_{1,-7})) - \chi_q(L(Y_{1,-1})) \chi_q(L( Y_{2,-8}Y_{2,-6}Y_{2,-4} )) \chi_q(L( Y_{1,-11}Y_{1,-9} )) \\
& - \chi_q( L( Y_{1,-5}Y_{1,-3}Y_{1,-1} ) ). 
\end{align*}
Computing the $q$-characters of evaluation modules on the right hand side and apply the Weyl character formula, we obtain that the restriction of $L(m)$ to $U_q(\sl_3)$ decomposes as
\begin{align*}
& L(4\poid_1+4\poid_2) \oplus L(5\poid_1+2\poid_2) \oplus L(2\poid_1+5\poid_2) \oplus 2L(3\poid_1+3\poid_2) \oplus L(6\poid_2) \\
& \oplus 2L(4\poid_1+\poid_2) \oplus 2L(\poid_1+4\poid_2) \oplus 3L(2\poid_1+2\poid_2) \oplus L(3\poid_1) \oplus L(3\poid_2) \oplus 2L(\poid_1+\poid_2).
\end{align*}
The dimension of this representation is 700.
The image of an irreducible summand under $\pik$ is given in the following :
\begin{scriptsize}
\[
\begin{array}{|c|c|c|c|c|}
\hline
\text{Multiplicity} & \text{Irreducible Summand $V$} & \dim V & \pik([V]) & \qdimk V \\
\hline
1 & L(4 \poid_1 + 4 \poid_2) & 125 & -[L(3 \La_1 + 3 \La_2)] & -4.41147... \\
1 & L(5 \poid_1 + 2 \poid_2) & 81 & 0 & 0 \\
1 & L(2 \poid_1 + 5 \poid_2) & 81 & 0 & 0 \\
2 & L(3 \poid_1 + 3 \poid_2) & 64 & [L(3 \La_1 + 3 \La_2)] & 4.41147... \\
1 & L(6 \poid_2) & 28 & L(6 \La_2) & 1 \\
2 & L(\poid_1 + 4 \poid_2) & 35 & [L(\La_0 + \La_1 + 4 \La_2)] & 5.41147... \\
2 & L(4 \poid_1 + \poid_2) & 35 & [L(\La_0 + 4 \La_1 + \La_2)] & 5.41147... \\
3 & L(2 \poid_1 + 2 \poid_2) & 27 & [L(2 \La_0 + 2 \La_1 + 2 \La_2)] & 8.63816... \\
1 & L(3 \poid_1) & 10 & [L(3 \La_0 + 3 \La_1)] & 4.41147... \\
1 & L(3 \poid_2) & 10 & [L(3 \La_0 + 3 \La_2)] & 4.41147... \\
2 & L(\poid_1 + \poid_2) & 8 & [L(4 \La_0 + \La_1 + \La_2)] & 5.41147... \\
\hline
\end{array}
\]
\end{scriptsize}
Let us verify the computation in the first row above explicitly, in accordance with \eqref{eq:Valcove}:
\[
\pik([L(4\poid_1+4\poid_2)]) = -[L(3\La_1+3\La_2)].
\]
Let $\la=-2\La_0+4\La_1+4\La_2 \in \Pk$ be the unique affine weight of level $\levk=6$ such that $\ol{\la} = 4\poid_1+4\poid_2$.
Since $\la$ is not an element of $\Pkp$, we employ the affine Weyl group action, given in \eqref{eq:fundamental_reflection_affine}, as follows:
\[
\begin{aligned}
s_0\cdot \la &= s_0\cdot (-2\La_0+4\La_1+4\La_2) \\
&= s_0(-2\La_0+4\La_1+4\La_2+\rho) - \rho \\
&= s_0(-\La_0+5\La_1+5\La_2) - \rho \\
&= (-\La_0+5\La_1+5\La_2) + \alpha_0 - \rho \pmod {\Z \nullroot} \\
&= (-\La_0+5\La_1+5\La_2) + (2\La_0-\La_1-\La_2) - (\La_0+\La_1+\La_2) \\
&= 3\La_1+3\La_2\in \Pkp.
\end{aligned}
\]
Let us consider the result in the second row:
\begin{equation}
\pik([L(5\poid_1 + 2\poid_2)]) = 0.
\label{eq:ex_pik}
\end{equation}
To verify this, 
let $\la=-\La_0+5\La_1+2\La_2 \in \Pk$ and $\ol{\la} = 5\poid_1+2\poid_2$.
We compute:
\[
\begin{aligned}
s_0\cdot \la &= s_0\cdot (-\La_0+5\La_1+2\La_2) \\
&= s_0(-\La_0+5\La_1+2\La_2+\rho) - \rho \\
&= -\La_0+5\La_1+2\La_2 \\
&= \la.
\end{aligned}
\]
Thus, $\la$ has a non-trivial isotropy subgroup under the shifted action of $W$, implying that it is null.
This confirms \eqref{eq:ex_pik}, from \eqref{eq:Valcove}.

We have
\[
\begin{aligned}
\phik([L(m)]) = &\ [L(3 \La_{0} + 3 \La_{1})] + [L(6 \La_{2})] + 2[L(4 \La_{0} + \La_{1} + \La_{2})] \\
& + 2[L(\La_{0} + 4 \La_{1} + \La_{2})] + 3[L(2 \La_{0} + 2 \La_{1} + 2 \La_{2})] \\
& + [L(3 \La_{0} + 3 \La_{2})] + [L(3 \La_{1} + 3 \La_{2})] + 2[L(\La_{0} + \La_{1} + 4 \La_{2})],
\end{aligned}
\]
which is positive, as expected in \conjref{conj:main1}.

\subsection{An example in type $B_3$}
We revisit the second case of \exref{Ex: AB 1}.
So, let \(\gzero\) be of type \(B_3\), with level \(\levk = 3\) and \(\ell = \levk - 1 = 2\).
In this case, $K_0(\Ctwo)$ is a cluster algebra, but not of finite cluster type.
Hence, there are infinitely many cluster variables.
The initial seed consists of 18 KR modules: 13 exchangeable and 5 frozen.
For the KR modules in the initial seed, we write
\begin{align*}
\xymatrix@C=0.7pc @R=.2pc{
x_{1}=[M_{5, 4}], & x_{2}=[M_{5, 0}], & x_{3} =[M_{5, -4}],\\
x_{4}=[M_{4, 2}], & x_{5}=[M_{4, -2}], & x_{6}=[M_{4, -6}],\\
x_{7} = [M_{3, 1}], & x_{8} = [M_{3, -1}], & x_{9} = [M_{3, -3}],\\ 
x_{10} = [M_{3, -5}], & x_{11}=[M_{3, -7}], & x_{12}=[M_{3, -9}],\\ 
x_{13} = [M_{2, 0}], & x_{14} = [M_{2, -4}], &x_{15} = [M_{2, -8}], \\
x_{16} = [M_{1, -2}], & x_{17} = [M_{1, -6}], &x_{18} = [M_{1, -10}].
}
\end{align*}
One of the cluster variables not in the initial seed is the class of the simple module
\[
M = L(Y_{1, 0}Y_{1, 4}Y_{2, -6}Y_{3, -3}Y_{3, -1}Y_{3, 1})
\]
of dimension $47880$. 
We can write $[M]\in K_0(\Ctwo)$ as a rational function in $x_i$ as follows:
\[
[M] = \frac{x_{10} x_3 x_5 x_7 + x_{10} x_3 x_4 x_9 + x_2 x_6 x_8 x_9}{x_5 x_8}.
\]
One can check that as a $\uqgzero$-module, $M$ decomposes as
\[
\begin{aligned}
& L(\poid_3) \oplus 6 L(3\poid_3) \oplus  L(5\poid_3) \oplus 5 L(\poid_2 + \poid_3) \oplus 5 L(\poid_2 + 3\poid_3) \\
&\oplus 3 L(2\poid_2 + \poid_3) \oplus  L(2\poid_2 + 3\poid_3) \oplus 5 L(\poid_1 + \poid_3) \oplus 6 L(\poid_1 + 3\poid_3) \\
&\oplus  L(\poid_1 + 5\poid_3) \oplus 10 L(\poid_1 + \poid_2 + \poid_3) \oplus 2 L(\poid_1 + \poid_2 + 3\poid_3) \\
&\oplus 3 L(\poid_1 + 2\poid_2 + \poid_3) \oplus 4 L(2\poid_1 + \poid_3) \oplus 5 L(2\poid_1 + 3\poid_3) \\
&\oplus 4 L(2\poid_1 + \poid_2 + \poid_3) \oplus  L(2\poid_1 + \poid_2 + 3\poid_3) \oplus 3 L(3\poid_1 + \poid_3) \\
&\oplus  L(3\poid_1 + 3\poid_3) \oplus  L(3\poid_1 + \poid_2 + \poid_3) \oplus  L(4\poid_1 + \poid_3).
\end{aligned}
\]
The image of an irreducible summand under $\pik$ is given in the following :
\begin{scriptsize}
\[
\begin{array}{|c|c|c|c|c|}
\hline
\text{Multiplicity} & \text{Irreducible Summand $V$} & \dim V & \pik([V]) & \qdimk V \\
\hline
1 & L(\poid_3) & 8 & [L(2\La_0 + \La_3)] & 3.62451... \\
2 & L(3\poid_3) & 112 & [L(3\La_3)] & 2.77408... \\
3 & L(\poid_2 + \poid_3) & 112 & [L(\La_2 + \La_3)] & 5.12583... \\
1 & L(\poid_2 + 3\poid_3) & 1008 & -[L(3\La_3)] & -2.77408... \\
2 & L(2\poid_2 + \poid_3) & 720 & -[L(\La_2 + \La_3)] & -5.12583... \\
2 & L(\poid_1 + \poid_3) & 48 & [L(\La_0 + \La_1 + \La_3)] & 6.69722... \\
3 & L(\poid_1 + 3\poid_3) & 560 & 0 & 0 \\
5 & L(\poid_1 + \poid_2 + \poid_3) & 512 & 0 & 0 \\
1 & L(\poid_1 + \poid_2 + 3\poid_3) & 4096 & 0 & 0 \\
2 & L(\poid_1 + 2\poid_2 + \poid_3) & 2800 & -[L(\La_0 + \La_1 + \La_3)] & -6.69722... \\
2 & L(2\poid_1 + \poid_3) & 168 & [L(2\La_1 + \La_3)] & 3.62451... \\
3 & L(2\poid_1 + 3\poid_3) & 1728 & 0 & 0 \\
3 & L(2\poid_1 + \poid_2 + \poid_3) & 1512 & -[L(2\La_1 + \La_3)] & -3.62451... \\
1 & L(2\poid_1 + \poid_2 + 3\poid_3) & 11088 & [L(\La_0 + \La_1 + \La_3)] & 6.69722... \\
2 & L(3\poid_1 + \poid_3) & 448 & 0 & 0 \\
1 & L(3\poid_1 + 3\poid_3) & 4200 & [L(2\La_1 + \La_3)] & 3.62451... \\
1 & L(3\poid_1 + \poid_2 + \poid_3) & 3584 & 0 & 0 \\
1 & L(4\poid_1 + \poid_3) & 1008 & 0 & 0 \\
\hline
\end{array}
\]
\end{scriptsize}
We find that  
\[
\begin{aligned}
\phik([M]) = [L(2\La_0 + \La_3)] + [L(3\La_3)] + [L(\La_2 + \La_3)] + [L(\La_0 + \La_1 + \La_3)],
\end{aligned}
\]
which, as predicted by \conjref{conj:main1}, is positive.

\appendix

\section{List of cluster variables in $\Cone$ for exceptional types}\label{app:A}
We provide a complete list of cluster variables of \(\Cone\) for exceptional types, along with the corresponding monomials and their images in the Verlinde ring here.
An extended version of the list, including expressions of cluster variables as rational functions of the initial variables, is available at \url{https://github.com/chlee-0/verlinde_cluster}. For instance, for type \(E_6\), refer to \url{https://github.com/chlee-0/verlinde_cluster/blob/main/cluster_variables_E6.txt}, among others.
As before, we choose the height function as in \eqref{eq:height-function}.

\begin{scriptsize}

\begin{longtable}{|c|c|c|}
\hline
\textbf{Cluster variable} & \textbf{Dominant monomial} & \textbf{Verlinde image} \\ \hline
$x_{1}$ & $Y_{1, -1}$ & $V_1$ \\ \hline
$x_{2}$ & $Y_{1, -3}Y_{1, -1}$ & $V_3$ \\ \hline
$x_{3}$ & $Y_{3, 0}$ & $V_2 + V_6$ \\ \hline
$x_{4}$ & $Y_{3, -2}Y_{3, 0}$ & $V_8$ \\ \hline
$x_{5}$ & $Y_{4, 1}$ & $V_0 + V_4 + 2  V_5$ \\ \hline
$x_{6}$ & $Y_{4, -1}Y_{4, 1}$ & $V_0$ \\ \hline
$x_{7}$ & $Y_{2, 2}$ & $V_0 + V_5$ \\ \hline
$x_{8}$ & $Y_{2, 0}Y_{2, 2}$ & $V_0$ \\ \hline
$x_{9}$ & $Y_{5, 2}$ & $V_1 + V_7$ \\ \hline
$x_{10}$ & $Y_{5, 0}Y_{5, 2}$ & $V_3$ \\ \hline
$x_{11}$ & $Y_{6, 3}$ & $V_2$ \\ \hline
$x_{12}$ & $Y_{6, 1}Y_{6, 3}$ & $V_8$ \\ \hline
$x_{13}$ & $Y_{3, -2}Y_{4, 1}$ & $V_2$ \\ \hline
$x_{14}$ & $Y_{5, 0}Y_{6, 3}$ & $V_4$ \\ \hline
$x_{15}$ & $Y_{2, 2}Y_{3, -2}Y_{6, 3}$ & $V_7 + V_8$ \\ \hline
$x_{16}$ & $Y_{3, -2}Y_{6, 3}$ & $V_7 + V_8$ \\ \hline
$x_{17}$ & $Y_{1, -3}Y_{6, 3}$ & $V_4$ \\ \hline
$x_{18}$ & $Y_{1, -3}Y_{2, 2}Y_{3, -2}Y_{6, 3}$ & $V_2 + V_6$ \\ \hline
$x_{19}$ & $Y_{2, 2}Y_{4, -1}Y_{6, 3}$ & $V_2$ \\ \hline
$x_{20}$ & $Y_{2, 2}Y_{3, -2}Y_{4, -1}Y_{5, 2}Y_{6, 3}$ & $V_2 + V_6$ \\ \hline
$x_{21}$ & $Y_{1, -3}Y_{2, 2}Y_{3, -2}Y_{4, -1}Y_{5, 2}Y_{6, 3}$ & $V_0 + V_4 + 2  V_5$ \\ \hline
$x_{22}$ & $Y_{1, -3}Y_{2, 2}Y_{4, -1}Y_{5, 2}$ & $V_3 + V_6$ \\ \hline
$x_{23}$ & $Y_{2, 2}Y_{3, -2}$ & $V_2$ \\ \hline
$x_{24}$ & $Y_{2, 2}Y_{3, -2}Y_{4, -1}Y_{5, 2}$ & $V_4 + V_5$ \\ \hline
$x_{25}$ & $Y_{3, -2}Y_{5, 2}$ & $V_4$ \\ \hline
$x_{26}$ & $Y_{1, -3}$ & $V_1$ \\ \hline
$x_{27}$ & $Y_{6, 1}$ & $V_2$ \\ \hline
$x_{28}$ & $Y_{3, -2}$ & $V_2 + V_6$ \\ \hline
$x_{29}$ & $Y_{4, -1}Y_{5, 2}$ & $V_1$ \\ \hline
$x_{30}$ & $Y_{2, 0}$ & $V_0 + V_5$ \\ \hline
$x_{31}$ & $Y_{4, -1}$ & $V_0 + V_4 + 2  V_5$ \\ \hline
$x_{32}$ & $Y_{4, -1}Y_{6, 3}$ & $V_2 + V_6$ \\ \hline
$x_{33}$ & $Y_{2, 2}Y_{4, -1}$ & $V_0 + V_5$ \\ \hline
$x_{34}$ & $Y_{1, -3}Y_{3, 0}$ & $V_4$ \\ \hline
$x_{35}$ & $Y_{1, -3}Y_{2, 2}$ & $V_1$ \\ \hline
$x_{36}$ & $Y_{2, 2}Y_{4, -1}Y_{5, 2}$ & $V_1$ \\ \hline
$x_{37}$ & $Y_{1, -3}Y_{2, 2}Y_{6, 3}$ & $V_4 + V_5$ \\ \hline
$x_{38}$ & $Y_{1, -3}Y_{2, 2}Y_{5, 2}$ & $V_2 + V_3 + 2  V_6$ \\ \hline
$x_{39}$ & $Y_{2, 2}Y_{3, -2}Y_{5, 2}$ & $V_4 + V_5$ \\ \hline
$x_{40}$ & $Y_{1, -3}Y_{5, 2}$ & $V_3 + V_6$ \\ \hline
$x_{41}$ & $Y_{1, -3}Y_{2, 2}Y_{4, -1}Y_{6, 3}$ & $V_4 + V_5$ \\ \hline
$x_{42}$ & $Y_{2, 2}Y_{3, -2}Y_{4, -1}Y_{6, 3}$ & $V_1 + 2  V_7 + V_8$ \\ \hline
$x_{43}$ & $Y_{1, -3}Y_{2, 2}Y_{4, -1}Y_{5, 2}Y_{6, 3}$ & $V_1 + V_7$ \\ \hline
$x_{44}$ & $Y_{1, -3}Y_{2, 2}Y_{3, -2}Y_{5, 2}Y_{6, 3}$ & $V_0 + V_4 + 2  V_5$ \\ \hline
$x_{45}$ & $Y_{1, -3}Y_{2, 2}Y_{3, -2}Y_{5, 2}$ & $V_1 + V_7$ \\ \hline
$x_{46}$ & $Y_{5, 0}$ & $V_1 + V_7$ \\ \hline
$x_{47}$ & $Y_{1, -3}Y_{4, 1}$ & $V_1 + V_7$ \\ \hline
$x_{48}$ & $Y_{1, -3}Y_{2, 2}Y_{2, 2}Y_{3, -2}Y_{4, -1}Y_{5, 2}Y_{6, 3}$ & $V_0 + V_4 + 2  V_5$ \\ \hline
\caption{Cluster variables with corresponding dominant monomials and their images under $\phik$ in $\Cone$ of type $E_6$.}
\label{tab:E6_cluster} 
\end{longtable}

\begin{longtable}{|c|c|c|}
\hline
\textbf{Cluster variable} & \textbf{Dominant monomial} & \textbf{Verlinde image} \\ \hline
$x_{1}$ & $Y_{1, -1}$ & $V_0 + V_2$ \\ \hline
$x_{2}$ & $Y_{1, -3}Y_{1, -1}$ & $V_0$ \\ \hline
$x_{3}$ & $Y_{3, 0}$ & $V_0 + 2  V_2 + V_4$ \\ \hline
$x_{4}$ & $Y_{3, -2}Y_{3, 0}$ & $V_0$ \\ \hline
$x_{5}$ & $Y_{4, 1}$ & $2  V_0 + 3  V_2 + 3  V_4 + V_5$ \\ \hline
$x_{6}$ & $Y_{4, -1}Y_{4, 1}$ & $V_0$ \\ \hline
$x_{7}$ & $Y_{2, 2}$ & $V_1 + V_3$ \\ \hline
$x_{8}$ & $Y_{2, 0}Y_{2, 2}$ & $V_5$ \\ \hline
$x_{9}$ & $Y_{5, 2}$ & $2  V_1 + 2  V_3$ \\ \hline
$x_{10}$ & $Y_{5, 0}Y_{5, 2}$ & $V_5$ \\ \hline
$x_{11}$ & $Y_{6, 3}$ & $V_0 + V_2 + V_4$ \\ \hline
$x_{12}$ & $Y_{6, 1}Y_{6, 3}$ & $V_0$ \\ \hline
$x_{13}$ & $Y_{7, 4}$ & $V_1$ \\ \hline
$x_{14}$ & $Y_{7, 2}Y_{7, 4}$ & $V_5$ \\ \hline
$x_{15}$ & $Y_{5, 0}Y_{6, 3}$ & $V_1$ \\ \hline
$x_{16}$ & $Y_{5, 0}Y_{7, 4}$ & $V_2 + V_4 + V_5$ \\ \hline
$x_{17}$ & $Y_{3, -2}Y_{4, 1}$ & $V_0 + V_2$ \\ \hline
$x_{18}$ & $Y_{2, 2}Y_{3, -2}Y_{7, 4}$ & $V_0 + V_2$ \\ \hline
$x_{19}$ & $Y_{1, -3}Y_{2, 2}Y_{7, 4}$ & $V_0 + V_2 + V_4$ \\ \hline
$x_{20}$ & $Y_{1, -3}Y_{7, 4}$ & $V_1$ \\ \hline
$x_{21}$ & $Y_{1, -3}Y_{2, 2}$ & $V_1$ \\ \hline
$x_{22}$ & $Y_{6, 1}Y_{7, 4}$ & $V_1$ \\ \hline
$x_{23}$ & $Y_{1, -3}Y_{2, 2}Y_{3, -2}Y_{5, 2}Y_{7, 4}$ & $2  V_1 + 2  V_3$ \\ \hline
$x_{24}$ & $Y_{1, -3}Y_{2, 2}Y_{4, -1}Y_{5, 2}Y_{7, 4}$ & $V_1 + V_3$ \\ \hline
$x_{25}$ & $Y_{1, -3}Y_{5, 2}$ & $V_1 + V_3$ \\ \hline
$x_{26}$ & $Y_{1, -3}Y_{2, 2}Y_{3, -2}Y_{5, 2}$ & $V_0 + V_2 + V_4$ \\ \hline
$x_{27}$ & $Y_{2, 2}Y_{4, -1}Y_{5, 2}$ & $V_0 + V_2$ \\ \hline
$x_{28}$ & $Y_{3, -2}Y_{5, 2}$ & $V_1$ \\ \hline
$x_{29}$ & $Y_{1, -3}Y_{2, 2}Y_{4, -1}Y_{5, 2}Y_{6, 3}$ & $V_0 + 2  V_2 + V_4$ \\ \hline
$x_{30}$ & $Y_{1, -3}Y_{2, 2}Y_{4, -1}Y_{5, 2}$ & $V_0 + V_2$ \\ \hline
$x_{31}$ & $Y_{7, 2}$ & $V_1$ \\ \hline
$x_{32}$ & $Y_{1, -3}Y_{2, 2}Y_{2, 2}Y_{3, -2}Y_{4, -1}Y_{5, 2}Y_{6, 3}$ & $2  V_1 + 2  V_3$ \\ \hline
$x_{33}$ & $Y_{2, 2}Y_{3, -2}Y_{4, -1}Y_{5, 2}Y_{6, 3}$ & $V_0 + V_2 + V_4$ \\ \hline
$x_{34}$ & $Y_{1, -3}Y_{6, 3}$ & $V_0 + V_2$ \\ \hline
$x_{35}$ & $Y_{2, 2}Y_{4, -1}Y_{6, 3}$ & $V_1$ \\ \hline
$x_{36}$ & $Y_{1, -3}Y_{2, 2}Y_{3, -2}Y_{6, 3}$ & $V_1 + V_3$ \\ \hline
$x_{37}$ & $Y_{2, 2}Y_{3, -2}Y_{6, 3}$ & $V_1 + V_3$ \\ \hline
$x_{38}$ & $Y_{1, -3}Y_{2, 2}Y_{3, -2}Y_{6, 3}Y_{7, 4}$ & $V_0 + 2  V_2 + V_4$ \\ \hline
$x_{39}$ & $Y_{3, -2}Y_{6, 3}$ & $V_0 + V_2$ \\ \hline
$x_{40}$ & $Y_{3, -2}Y_{7, 4}$ & $V_1 + V_3$ \\ \hline
$x_{41}$ & $Y_{2, 0}$ & $V_1 + V_3$ \\ \hline
$x_{42}$ & $Y_{4, -1}Y_{7, 4}$ & $2  V_1 + 2  V_3$ \\ \hline
$x_{43}$ & $Y_{1, -3}$ & $V_0 + V_2$ \\ \hline
$x_{44}$ & $Y_{3, -2}$ & $V_0 + 2  V_2 + V_4$ \\ \hline
$x_{45}$ & $Y_{4, -1}$ & $2  V_0 + 3  V_2 + 3  V_4 + V_5$ \\ \hline
$x_{46}$ & $Y_{4, -1}Y_{5, 2}$ & $V_1$ \\ \hline
$x_{47}$ & $Y_{4, -1}Y_{6, 3}$ & $V_0 + V_2 + V_4$ \\ \hline
$x_{48}$ & $Y_{2, 2}Y_{4, -1}$ & $V_1 + V_3$ \\ \hline
$x_{49}$ & $Y_{1, -3}Y_{3, 0}$ & $V_0 + V_2$ \\ \hline
$x_{50}$ & $Y_{1, -3}Y_{2, 2}Y_{6, 3}$ & $2  V_1 + 2  V_3$ \\ \hline
$x_{51}$ & $Y_{1, -3}Y_{4, 1}$ & $V_0 + 2  V_2 + V_4$ \\ \hline
$x_{52}$ & $Y_{1, -3}Y_{2, 2}Y_{5, 2}$ & $2  V_0 + 3  V_2 + 3  V_4 + V_5$ \\ \hline
$x_{53}$ & $Y_{6, 1}$ & $V_0 + V_2 + V_4$ \\ \hline
$x_{54}$ & $Y_{1, -3}Y_{2, 2}Y_{3, -2}Y_{4, -1}Y_{5, 2}Y_{6, 3}Y_{7, 4}$ & $2  V_1 + 2  V_3$ \\ \hline
$x_{55}$ & $Y_{1, -3}Y_{1, -3}Y_{2, 2}Y_{2, 2}Y_{3, -2}Y_{4, -1}Y_{5, 2}Y_{6, 3}Y_{7, 4}$ & $2  V_0 + 3  V_2 + 3  V_4 + V_5$ \\ \hline
$x_{56}$ & $Y_{1, -3}Y_{2, 2}Y_{4, -1}Y_{6, 3}Y_{7, 4}$ & $V_0 + V_2 + V_4$ \\ \hline
$x_{57}$ & $Y_{1, -3}Y_{2, 2}Y_{2, 2}Y_{3, -2}Y_{4, -1}Y_{5, 2}Y_{7, 4}$ & $V_0 + 2  V_2 + V_4$ \\ \hline
$x_{58}$ & $Y_{2, 2}Y_{4, -1}Y_{7, 4}$ & $V_0 + V_2$ \\ \hline
$x_{59}$ & $Y_{5, 0}$ & $2  V_1 + 2  V_3$ \\ \hline
$x_{60}$ & $Y_{2, 2}Y_{3, -2}$ & $V_1$ \\ \hline
$x_{61}$ & $Y_{1, -3}Y_{2, 2}Y_{3, -2}Y_{7, 4}$ & $V_0 + V_2 + V_4$ \\ \hline
$x_{62}$ & $Y_{1, -3}Y_{2, 2}Y_{3, -2}Y_{4, -1}Y_{5, 2}Y_{6, 3}$ & $V_0 + 2  V_2 + V_4$ \\ \hline
$x_{63}$ & $Y_{1, -3}Y_{2, 2}Y_{2, 2}Y_{3, -2}Y_{3, -2}Y_{4, -1}Y_{5, 2}Y_{6, 3}Y_{7, 4}$ & $2  V_0 + 3  V_2 + 3  V_4 + V_5$ \\ \hline
$x_{64}$ & $Y_{2, 2}Y_{3, -2}Y_{4, -1}Y_{7, 4}$ & $2  V_0 + 3  V_2 + 3  V_4 + V_5$ \\ \hline
$x_{65}$ & $Y_{2, 2}Y_{3, -2}Y_{4, -1}Y_{5, 2}$ & $V_0 + V_2 + V_4$ \\ \hline
$x_{66}$ & $Y_{2, 2}Y_{3, -2}Y_{4, -1}Y_{6, 3}$ & $2  V_1 + 2  V_3$ \\ \hline
$x_{67}$ & $Y_{1, -3}Y_{2, 2}Y_{4, -1}Y_{7, 4}$ & $V_0 + 2  V_2 + V_4$ \\ \hline
$x_{68}$ & $Y_{1, -3}Y_{2, 2}Y_{4, -1}Y_{6, 3}$ & $V_1 + V_3$ \\ \hline
$x_{69}$ & $Y_{2, 2}Y_{3, -2}Y_{5, 2}$ & $V_0 + 2  V_2 + V_4$ \\ \hline
$x_{70}$ & $Y_{2, 2}Y_{3, -2}Y_{4, -1}Y_{6, 3}Y_{7, 4}$ & $V_0 + 2  V_2 + V_4$ \\ \hline
$x_{71}$ & $Y_{2, 2}Y_{3, -2}Y_{4, -1}Y_{5, 2}Y_{7, 4}$ & $V_1 + V_3$ \\ \hline
$x_{72}$ & $Y_{1, -3}Y_{2, 2}Y_{2, 2}Y_{3, -2}Y_{4, -1}Y_{6, 3}Y_{7, 4}$ & $2  V_1 + 2  V_3$ \\ \hline
$x_{73}$ & $Y_{1, -3}Y_{2, 2}Y_{2, 2}Y_{3, -2}Y_{4, -1}Y_{4, -1}Y_{5, 2}Y_{6, 3}Y_{7, 4}$ & $2  V_0 + 3  V_2 + 3  V_4 + V_5$ \\ \hline
$x_{74}$ & $Y_{1, -3}Y_{2, 2}Y_{3, -2}Y_{4, -1}Y_{5, 2}Y_{7, 4}$ & $2  V_1 + 2  V_3$ \\ \hline
$x_{75}$ & $Y_{1, -3}Y_{2, 2}Y_{3, -2}Y_{5, 2}Y_{6, 3}$ & $2  V_0 + 3  V_2 + 3  V_4 + V_5$ \\ \hline
$x_{76}$ & $Y_{1, -3}Y_{2, 2}Y_{2, 2}Y_{3, -2}Y_{4, -1}Y_{5, 2}Y_{6, 3}Y_{7, 4}$ & $2  V_0 + 3  V_2 + 3  V_4 + V_5$ \\ \hline
$x_{77}$ & $Y_{1, -3}Y_{2, 2}Y_{3, -2}Y_{4, -1}Y_{6, 3}Y_{7, 4}$ & $2  V_0 + 3  V_2 + 3  V_4 + V_5$ \\ \hline
\caption{Cluster variables with corresponding dominant monomials and their images under $\phik$ in $\Cone$ of type $E_7$.}
\label{tab:E7_cluster} 
\end{longtable}

\begin{longtable}{|c|c|c|}
\hline
\textbf{Cluster variable} & \textbf{Dominant monomial} & \textbf{Verlinde image} \\ \hline
$x_{1}$ & $Y_{1, -1}$ & $V_0 + V_1 + V_2$ \\ \hline
$x_{2}$ & $Y_{1, -3}Y_{1, -1}$ & $V_0$ \\ \hline
$x_{3}$ & $Y_{3, 0}$ & $2  V_0 + 3  V_1 + 4  V_2$ \\ \hline
$x_{4}$ & $Y_{3, -2}Y_{3, 0}$ & $V_0$ \\ \hline
$x_{5}$ & $Y_{4, 1}$ & $8  V_0 + 8  V_1 + 12  V_2$ \\ \hline
$x_{6}$ & $Y_{4, -1}Y_{4, 1}$ & $V_0$ \\ \hline
$x_{7}$ & $Y_{2, 2}$ & $V_0 + 2  V_1 + 2  V_2$ \\ \hline
$x_{8}$ & $Y_{2, 0}Y_{2, 2}$ & $V_0$ \\ \hline
$x_{9}$ & $Y_{5, 2}$ & $4  V_0 + 5  V_1 + 6  V_2$ \\ \hline
$x_{10}$ & $Y_{5, 0}Y_{5, 2}$ & $V_0$ \\ \hline
$x_{11}$ & $Y_{6, 3}$ & $2  V_0 + 3  V_1 + 3  V_2$ \\ \hline
$x_{12}$ & $Y_{6, 1}Y_{6, 3}$ & $V_0$ \\ \hline
$x_{13}$ & $Y_{7, 4}$ & $V_0 + V_1 + 2  V_2$ \\ \hline
$x_{14}$ & $Y_{7, 2}Y_{7, 4}$ & $V_0$ \\ \hline
$x_{15}$ & $Y_{8, 5}$ & $V_0 + V_2$ \\ \hline
$x_{16}$ & $Y_{8, 3}Y_{8, 5}$ & $V_0$ \\ \hline
$x_{17}$ & $Y_{8, 3}$ & $V_0 + V_2$ \\ \hline
$x_{18}$ & $Y_{7, 2}$ & $V_0 + V_1 + 2  V_2$ \\ \hline
$x_{19}$ & $Y_{6, 1}$ & $2  V_0 + 3  V_1 + 3  V_2$ \\ \hline
$x_{20}$ & $Y_{2, 2}Y_{4, -1}Y_{5, 2}$ & $V_0 + V_1 + V_2$ \\ \hline
$x_{21}$ & $Y_{1, -3}Y_{3, 0}$ & $V_0 + V_1 + V_2$ \\ \hline
$x_{22}$ & $Y_{1, -3}Y_{2, 2}Y_{5, 2}$ & $8  V_0 + 8  V_1 + 12  V_2$ \\ \hline
$x_{23}$ & $Y_{1, -3}Y_{2, 2}$ & $V_0 + V_2$ \\ \hline
$x_{24}$ & $Y_{1, -3}Y_{5, 2}$ & $V_0 + 2  V_1 + 2  V_2$ \\ \hline
$x_{25}$ & $Y_{7, 2}Y_{8, 5}$ & $V_0 + V_2$ \\ \hline
$x_{26}$ & $Y_{6, 1}Y_{8, 5}$ & $V_0 + V_1 + 2  V_2$ \\ \hline
$x_{27}$ & $Y_{1, -3}Y_{2, 2}Y_{8, 5}$ & $V_0 + V_1 + 2  V_2$ \\ \hline
$x_{28}$ & $Y_{1, -3}Y_{4, 1}$ & $2  V_0 + 3  V_1 + 4  V_2$ \\ \hline
$x_{29}$ & $Y_{3, -2}Y_{4, 1}$ & $V_0 + V_1 + V_2$ \\ \hline
$x_{30}$ & $Y_{1, -3}Y_{8, 5}$ & $V_0 + V_2$ \\ \hline
$x_{31}$ & $Y_{5, 0}Y_{8, 5}$ & $2  V_0 + 3  V_1 + 3  V_2$ \\ \hline
$x_{32}$ & $Y_{6, 1}Y_{7, 4}$ & $V_0 + V_2$ \\ \hline
$x_{33}$ & $Y_{5, 0}Y_{7, 4}$ & $V_0 + V_1 + 2  V_2$ \\ \hline
$x_{34}$ & $Y_{2, 0}$ & $V_0 + 2  V_1 + 2  V_2$ \\ \hline
$x_{35}$ & $Y_{3, -2}Y_{8, 5}$ & $V_0 + 2  V_1 + 2  V_2$ \\ \hline
$x_{36}$ & $Y_{4, -1}Y_{8, 5}$ & $4  V_0 + 5  V_1 + 6  V_2$ \\ \hline
$x_{37}$ & $Y_{3, -2}$ & $2  V_0 + 3  V_1 + 4  V_2$ \\ \hline
$x_{38}$ & $Y_{4, -1}Y_{7, 4}$ & $2  V_0 + 3  V_1 + 3  V_2$ \\ \hline
$x_{39}$ & $Y_{5, 0}Y_{6, 3}$ & $V_0 + V_2$ \\ \hline
$x_{40}$ & $Y_{4, -1}Y_{6, 3}$ & $V_0 + V_1 + 2  V_2$ \\ \hline
$x_{41}$ & $Y_{2, 2}Y_{3, -2}Y_{4, -1}Y_{8, 5}$ & $8  V_0 + 8  V_1 + 12  V_2$ \\ \hline
$x_{42}$ & $Y_{2, 2}Y_{4, -1}$ & $V_0 + 2  V_1 + 2  V_2$ \\ \hline
$x_{43}$ & $Y_{2, 2}Y_{3, -2}Y_{4, -1}Y_{7, 4}$ & $4  V_0 + 5  V_1 + 6  V_2$ \\ \hline
$x_{44}$ & $Y_{1, -3}$ & $V_0 + V_1 + V_2$ \\ \hline
$x_{45}$ & $Y_{1, -3}Y_{2, 2}Y_{4, -1}Y_{8, 5}$ & $2  V_0 + 3  V_1 + 4  V_2$ \\ \hline
$x_{46}$ & $Y_{2, 2}Y_{3, -2}Y_{4, -1}Y_{6, 3}$ & $2  V_0 + 3  V_1 + 3  V_2$ \\ \hline
$x_{47}$ & $Y_{4, -1}Y_{5, 2}$ & $V_0 + V_2$ \\ \hline
$x_{48}$ & $Y_{2, 2}Y_{3, -2}Y_{4, -1}Y_{5, 2}$ & $V_0 + V_1 + 2  V_2$ \\ \hline
$x_{49}$ & $Y_{1, -3}Y_{2, 2}Y_{4, -1}Y_{7, 4}$ & $V_0 + 2  V_1 + 2  V_2$ \\ \hline
$x_{50}$ & $Y_{1, -3}Y_{2, 2}Y_{3, -2}Y_{4, -1}Y_{7, 4}Y_{8, 5}$ & $8  V_0 + 8  V_1 + 12  V_2$ \\ \hline
$x_{51}$ & $Y_{1, -3}Y_{2, 2}Y_{3, -2}Y_{4, -1}Y_{6, 3}Y_{8, 5}$ & $4  V_0 + 5  V_1 + 6  V_2$ \\ \hline
$x_{52}$ & $Y_{2, 2}Y_{4, -1}Y_{8, 5}$ & $V_0 + V_1 + V_2$ \\ \hline
$x_{53}$ & $Y_{2, 2}Y_{3, -2}Y_{4, -1}Y_{7, 4}Y_{8, 5}$ & $2  V_0 + 3  V_1 + 4  V_2$ \\ \hline
$x_{54}$ & $Y_{1, -3}Y_{2, 2}Y_{3, -2}Y_{4, -1}Y_{5, 2}Y_{8, 5}$ & $2  V_0 + 3  V_1 + 3  V_2$ \\ \hline
$x_{55}$ & $Y_{2, 2}Y_{3, -2}$ & $V_0 + V_2$ \\ \hline
$x_{56}$ & $Y_{1, -3}Y_{2, 2}Y_{3, -2}Y_{8, 5}$ & $V_0 + V_1 + 2  V_2$ \\ \hline
$x_{57}$ & $Y_{2, 2}Y_{3, -2}Y_{4, -1}Y_{6, 3}Y_{8, 5}$ & $V_0 + 2  V_1 + 2  V_2$ \\ \hline
$x_{58}$ & $Y_{2, 2}Y_{3, -2}Y_{4, -1}Y_{5, 2}Y_{8, 5}$ & $V_0 + V_1 + V_2$ \\ \hline
$x_{59}$ & $Y_{3, -2}Y_{7, 4}$ & $V_0 + V_1 + V_2$ \\ \hline
$x_{60}$ & $Y_{3, -2}Y_{6, 3}$ & $V_0 + V_2$ \\ \hline
$x_{61}$ & $Y_{2, 2}Y_{3, -2}Y_{8, 5}$ & $V_0 + V_2$ \\ \hline
$x_{62}$ & $Y_{1, -3}Y_{2, 2}Y_{3, -2}Y_{4, -1}Y_{6, 3}Y_{7, 4}Y_{8, 5}$ & $2  V_0 + 3  V_1 + 3  V_2$ \\ \hline
$x_{63}$ & $Y_{1, -3}Y_{2, 2}Y_{3, -2}Y_{4, -1}Y_{5, 2}Y_{7, 4}Y_{8, 5}$ & $V_0 + 2  V_1 + 2  V_2$ \\ \hline
$x_{64}$ & $Y_{1, -3}Y_{2, 2}Y_{3, -2}Y_{7, 4}Y_{8, 5}$ & $V_0 + V_1 + 2  V_2$ \\ \hline
$x_{65}$ & $Y_{1, -3}Y_{2, 2}Y_{4, -1}Y_{6, 3}Y_{8, 5}$ & $V_0 + V_1 + V_2$ \\ \hline
$x_{66}$ & $Y_{1, -3}Y_{2, 2}Y_{4, -1}Y_{7, 4}Y_{8, 5}$ & $V_0 + V_1 + 2  V_2$ \\ \hline
$x_{67}$ & $Y_{1, -3}Y_{7, 4}$ & $V_0 + V_2$ \\ \hline
$x_{68}$ & $Y_{2, 2}Y_{4, -1}Y_{7, 4}$ & $V_0 + V_2$ \\ \hline
$x_{69}$ & $Y_{1, -3}Y_{1, -3}Y_{2, 2}Y_{2, 2}Y_{3, -2}Y_{4, -1}Y_{6, 3}Y_{7, 4}Y_{8, 5}$ & $2  V_0 + 3  V_1 + 4  V_2$ \\ \hline
$x_{70}$ & $Y_{1, -3}Y_{1, -3}Y_{2, 2}Y_{2, 2}Y_{3, -2}Y_{4, -1}Y_{5, 2}Y_{7, 4}Y_{8, 5}$ & $2  V_0 + 3  V_1 + 3  V_2$ \\ \hline
$x_{71}$ & $Y_{1, -3}Y_{2, 2}Y_{4, -1}Y_{5, 2}$ & $V_0 + V_2$ \\ \hline
$x_{72}$ & $Y_{1, -3}Y_{2, 2}Y_{3, -2}Y_{7, 4}$ & $V_0 + V_1 + V_2$ \\ \hline
$x_{73}$ & $Y_{1, -3}Y_{2, 2}Y_{2, 2}Y_{3, -2}Y_{4, -1}Y_{6, 3}Y_{7, 4}$ & $V_0 + 2  V_1 + 2  V_2$ \\ \hline
$x_{74}$ & $Y_{1, -3}Y_{2, 2}Y_{2, 2}Y_{3, -2}Y_{4, -1}Y_{5, 2}Y_{8, 5}$ & $V_0 + V_1 + 2  V_2$ \\ \hline
$x_{75}$ & $Y_{1, -3}Y_{2, 2}Y_{3, -2}Y_{6, 3}$ & $V_0 + V_1 + V_2$ \\ \hline
$x_{76}$ & $Y_{1, -3}Y_{2, 2}Y_{2, 2}Y_{3, -2}Y_{4, -1}Y_{5, 2}Y_{7, 4}$ & $V_0 + 2  V_1 + 2  V_2$ \\ \hline
$x_{77}$ & $Y_{1, -3}Y_{2, 2}Y_{2, 2}Y_{3, -2}Y_{3, -2}Y_{4, -1}Y_{5, 2}Y_{6, 3}Y_{7, 4}$ & $2  V_0 + 3  V_1 + 4  V_2$ \\ \hline
$x_{78}$ & $Y_{2, 2}Y_{4, -1}Y_{6, 3}$ & $V_0 + V_2$ \\ \hline
$x_{79}$ & $Y_{2, 2}Y_{3, -2}Y_{4, -1}Y_{5, 2}Y_{6, 3}$ & $V_0 + V_1 + 2  V_2$ \\ \hline
$x_{80}$ & $Y_{2, 2}Y_{3, -2}Y_{7, 4}$ & $V_0 + V_1 + V_2$ \\ \hline
$x_{81}$ & $Y_{2, 2}Y_{3, -2}Y_{4, -1}Y_{5, 2}Y_{7, 4}$ & $V_0 + V_1 + V_2$ \\ \hline
$x_{82}$ & $Y_{1, -3}Y_{2, 2}Y_{2, 2}Y_{3, -2}Y_{3, -2}Y_{4, -1}Y_{5, 2}Y_{6, 3}Y_{7, 4}Y_{8, 5}$ & $4  V_0 + 5  V_1 + 6  V_2$ \\ \hline
$x_{83}$ & $Y_{1, -3}Y_{2, 2}Y_{3, -2}Y_{4, -1}Y_{5, 2}Y_{6, 3}Y_{7, 4}$ & $2  V_0 + 3  V_1 + 3  V_2$ \\ \hline
$x_{84}$ & $Y_{3, -2}Y_{5, 2}$ & $V_0 + V_2$ \\ \hline
$x_{85}$ & $Y_{1, -3}Y_{2, 2}Y_{3, -2}Y_{4, -1}Y_{5, 2}Y_{6, 3}Y_{8, 5}$ & $V_0 + 2  V_1 + 2  V_2$ \\ \hline
$x_{86}$ & $Y_{1, -3}Y_{2, 2}Y_{3, -2}Y_{6, 3}Y_{8, 5}$ & $V_0 + 2  V_1 + 2  V_2$ \\ \hline
$x_{87}$ & $Y_{1, -3}Y_{2, 2}Y_{4, -1}Y_{5, 2}Y_{8, 5}$ & $V_0 + V_1 + V_2$ \\ \hline
$x_{88}$ & $Y_{1, -3}Y_{6, 3}$ & $V_0 + V_1 + V_2$ \\ \hline
$x_{89}$ & $Y_{1, -3}Y_{1, -3}Y_{2, 2}Y_{2, 2}Y_{3, -2}Y_{4, -1}Y_{5, 2}Y_{6, 3}Y_{8, 5}$ & $2  V_0 + 3  V_1 + 4  V_2$ \\ \hline
$x_{90}$ & $Y_{1, -3}Y_{2, 2}Y_{2, 2}Y_{3, -2}Y_{4, -1}Y_{6, 3}Y_{8, 5}$ & $V_0 + 2  V_1 + 2  V_2$ \\ \hline
$x_{91}$ & $Y_{1, -3}Y_{2, 2}Y_{2, 2}Y_{3, -2}Y_{3, -2}Y_{4, -1}Y_{5, 2}Y_{6, 3}Y_{8, 5}$ & $2  V_0 + 3  V_1 + 3  V_2$ \\ \hline
$x_{92}$ & $Y_{1, -3}Y_{2, 2}Y_{2, 2}Y_{2, 2}Y_{3, -2}Y_{3, -2}Y_{4, -1}Y_{4, -1}Y_{5, 2}Y_{6, 3}Y_{7, 4}Y_{8, 5}$ & $4  V_0 + 5  V_1 + 6  V_2$ \\ \hline
$x_{93}$ & $Y_{1, -3}Y_{1, -3}Y_{2, 2}Y_{2, 2}Y_{2, 2}Y_{3, -2}Y_{3, -2}Y_{4, -1}Y_{4, -1}Y_{5, 2}Y_{6, 3}Y_{7, 4}Y_{7, 4}Y_{8, 5}$ & $8  V_0 + 8  V_1 + 12  V_2$ \\ \hline
$x_{94}$ & $Y_{2, 2}Y_{3, -2}Y_{4, -1}Y_{6, 3}Y_{7, 4}$ & $V_0 + V_1 + 2  V_2$ \\ \hline
$x_{95}$ & $Y_{1, -3}Y_{2, 2}Y_{2, 2}Y_{3, -2}Y_{4, -1}Y_{4, -1}Y_{5, 2}Y_{6, 3}Y_{7, 4}$ & $2  V_0 + 3  V_1 + 3  V_2$ \\ \hline
$x_{96}$ & $Y_{1, -3}Y_{1, -3}Y_{2, 2}Y_{2, 2}Y_{3, -2}Y_{4, -1}Y_{4, -1}Y_{5, 2}Y_{6, 3}Y_{7, 4}Y_{8, 5}$ & $4  V_0 + 5  V_1 + 6  V_2$ \\ \hline
$x_{97}$ & $Y_{1, -3}Y_{2, 2}Y_{4, -1}Y_{6, 3}$ & $V_0 + V_1 + V_2$ \\ \hline
$x_{98}$ & $Y_{1, -3}Y_{2, 2}Y_{2, 2}Y_{3, -2}Y_{4, -1}Y_{4, -1}Y_{5, 2}Y_{6, 3}Y_{8, 5}$ & $2  V_0 + 3  V_1 + 4  V_2$ \\ \hline
$x_{99}$ & $Y_{1, -3}Y_{1, -3}Y_{2, 2}Y_{2, 2}Y_{2, 2}Y_{3, -2}Y_{3, -2}Y_{4, -1}Y_{4, -1}Y_{5, 2}Y_{6, 3}Y_{7, 4}Y_{8, 5}$ & $8  V_0 + 8  V_1 + 12  V_2$ \\ \hline
$x_{100}$ & $Y_{1, -3}Y_{2, 2}Y_{2, 2}Y_{3, -2}Y_{4, -1}Y_{7, 4}Y_{8, 5}$ & $2  V_0 + 3  V_1 + 3  V_2$ \\ \hline
$x_{101}$ & $Y_{1, -3}Y_{2, 2}Y_{2, 2}Y_{3, -2}Y_{3, -2}Y_{4, -1}Y_{6, 3}Y_{7, 4}Y_{8, 5}$ & $4  V_0 + 5  V_1 + 6  V_2$ \\ \hline
$x_{102}$ & $Y_{1, -3}Y_{2, 2}Y_{3, -2}Y_{4, -1}Y_{5, 2}Y_{7, 4}$ & $V_0 + 2  V_1 + 2  V_2$ \\ \hline
$x_{103}$ & $Y_{1, -3}Y_{2, 2}Y_{2, 2}Y_{3, -2}Y_{3, -2}Y_{4, -1}Y_{5, 2}Y_{7, 4}Y_{8, 5}$ & $2  V_0 + 3  V_1 + 4  V_2$ \\ \hline
$x_{104}$ & $Y_{1, -3}Y_{1, -3}Y_{2, 2}Y_{2, 2}Y_{3, -2}Y_{3, -2}Y_{4, -1}Y_{5, 2}Y_{6, 3}Y_{7, 4}Y_{8, 5}$ & $4  V_0 + 5  V_1 + 6  V_2$ \\ \hline
$x_{105}$ & $Y_{1, -3}Y_{2, 2}Y_{2, 2}Y_{3, -2}Y_{4, -1}Y_{4, -1}Y_{5, 2}Y_{6, 3}Y_{7, 4}Y_{8, 5}$ & $2  V_0 + 3  V_1 + 4  V_2$ \\ \hline
$x_{106}$ & $Y_{1, -3}Y_{1, -3}Y_{2, 2}Y_{2, 2}Y_{2, 2}Y_{3, -2}Y_{3, -2}Y_{4, -1}Y_{4, -1}Y_{5, 2}Y_{6, 3}Y_{6, 3}Y_{7, 4}Y_{8, 5}$ & $8  V_0 + 8  V_1 + 12  V_2$ \\ \hline
$x_{107}$ & $Y_{1, -3}Y_{1, -3}Y_{2, 2}Y_{2, 2}Y_{2, 2}Y_{3, -2}Y_{4, -1}Y_{4, -1}Y_{5, 2}Y_{6, 3}Y_{7, 4}Y_{8, 5}$ & $4  V_0 + 5  V_1 + 6  V_2$ \\ \hline
$x_{108}$ & $Y_{1, -3}Y_{2, 2}Y_{4, -1}Y_{6, 3}Y_{7, 4}$ & $V_0 + V_1 + 2  V_2$ \\ \hline
$x_{109}$ & $Y_{1, -3}Y_{2, 2}Y_{2, 2}Y_{3, -2}Y_{4, -1}Y_{6, 3}Y_{7, 4}Y_{8, 5}$ & $2  V_0 + 3  V_1 + 3  V_2$ \\ \hline
$x_{110}$ & $Y_{1, -3}Y_{2, 2}Y_{2, 2}Y_{3, -2}Y_{4, -1}Y_{5, 2}Y_{7, 4}Y_{8, 5}$ & $2  V_0 + 3  V_1 + 4  V_2$ \\ \hline
$x_{111}$ & $Y_{1, -3}Y_{2, 2}Y_{3, -2}Y_{5, 2}Y_{8, 5}$ & $2  V_0 + 3  V_1 + 3  V_2$ \\ \hline
$x_{112}$ & $Y_{1, -3}Y_{2, 2}Y_{3, -2}Y_{5, 2}$ & $V_0 + V_1 + 2  V_2$ \\ \hline
$x_{113}$ & $Y_{5, 0}$ & $4  V_0 + 5  V_1 + 6  V_2$ \\ \hline
$x_{114}$ & $Y_{4, -1}$ & $8  V_0 + 8  V_1 + 12  V_2$ \\ \hline
$x_{115}$ & $Y_{1, -3}Y_{2, 2}Y_{2, 2}Y_{3, -2}Y_{4, -1}Y_{5, 2}Y_{6, 3}$ & $2  V_0 + 3  V_1 + 3  V_2$ \\ \hline
$x_{116}$ & $Y_{1, -3}Y_{2, 2}Y_{3, -2}Y_{4, -1}Y_{5, 2}Y_{6, 3}$ & $V_0 + V_1 + 2  V_2$ \\ \hline
$x_{117}$ & $Y_{1, -3}Y_{1, -3}Y_{2, 2}Y_{2, 2}Y_{2, 2}Y_{3, -2}Y_{3, -2}Y_{4, -1}Y_{5, 2}Y_{6, 3}Y_{7, 4}Y_{8, 5}$ & $8  V_0 + 8  V_1 + 12  V_2$ \\ \hline
$x_{118}$ & $Y_{1, -3}Y_{1, -3}Y_{2, 2}Y_{2, 2}Y_{2, 2}Y_{3, -2}Y_{3, -2}Y_{4, -1}Y_{4, -1}Y_{5, 2}Y_{6, 3}Y_{7, 4}Y_{8, 5}Y_{8, 5}$ & $8  V_0 + 8  V_1 + 12  V_2$ \\ \hline
$x_{119}$ & $Y_{2, 2}Y_{3, -2}Y_{6, 3}$ & $V_0 + 2  V_1 + 2  V_2$ \\ \hline
$x_{120}$ & $Y_{1, -3}Y_{2, 2}Y_{3, -2}Y_{6, 3}Y_{7, 4}$ & $2  V_0 + 3  V_1 + 4  V_2$ \\ \hline
$x_{121}$ & $Y_{1, -3}Y_{2, 2}Y_{2, 2}Y_{3, -2}Y_{4, -1}Y_{5, 2}Y_{6, 3}Y_{8, 5}$ & $4  V_0 + 5  V_1 + 6  V_2$ \\ \hline
$x_{122}$ & $Y_{1, -3}Y_{1, -3}Y_{2, 2}Y_{2, 2}Y_{3, -2}Y_{4, -1}Y_{5, 2}Y_{6, 3}Y_{7, 4}Y_{8, 5}$ & $8  V_0 + 8  V_1 + 12  V_2$ \\ \hline
$x_{123}$ & $Y_{1, -3}Y_{2, 2}Y_{4, -1}Y_{5, 2}Y_{7, 4}$ & $V_0 + 2  V_1 + 2  V_2$ \\ \hline
$x_{124}$ & $Y_{1, -3}Y_{2, 2}Y_{4, -1}Y_{5, 2}Y_{6, 3}$ & $2  V_0 + 3  V_1 + 4  V_2$ \\ \hline
$x_{125}$ & $Y_{1, -3}Y_{1, -3}Y_{2, 2}Y_{2, 2}Y_{2, 2}Y_{3, -2}Y_{3, -2}Y_{4, -1}Y_{4, -1}Y_{5, 2}Y_{5, 2}Y_{6, 3}Y_{7, 4}Y_{8, 5}$ & $8  V_0 + 8  V_1 + 12  V_2$ \\ \hline
$x_{126}$ & $Y_{2, 2}Y_{3, -2}Y_{5, 2}$ & $2  V_0 + 3  V_1 + 4  V_2$ \\ \hline
$x_{127}$ & $Y_{1, -3}Y_{2, 2}Y_{6, 3}$ & $4  V_0 + 5  V_1 + 6  V_2$ \\ \hline
$x_{128}$ & $Y_{1, -3}Y_{2, 2}Y_{7, 4}$ & $2  V_0 + 3  V_1 + 3  V_2$ \\ \hline
$x_{129}$ & $Y_{1, -3}Y_{2, 2}Y_{3, -2}Y_{4, -1}Y_{6, 3}Y_{7, 4}$ & $2  V_0 + 3  V_1 + 4  V_2$ \\ \hline
$x_{130}$ & $Y_{1, -3}Y_{2, 2}Y_{2, 2}Y_{3, -2}Y_{4, -1}Y_{4, -1}Y_{6, 3}Y_{7, 4}Y_{8, 5}$ & $8  V_0 + 8  V_1 + 12  V_2$ \\ \hline
$x_{131}$ & $Y_{1, -3}Y_{2, 2}Y_{2, 2}Y_{3, -2}Y_{4, -1}Y_{4, -1}Y_{5, 2}Y_{7, 4}Y_{8, 5}$ & $4  V_0 + 5  V_1 + 6  V_2$ \\ \hline
$x_{132}$ & $Y_{1, -3}Y_{2, 2}Y_{2, 2}Y_{3, -2}Y_{4, -1}Y_{5, 2}Y_{6, 3}Y_{7, 4}$ & $8  V_0 + 8  V_1 + 12  V_2$ \\ \hline
$x_{133}$ & $Y_{1, -3}Y_{2, 2}Y_{3, -2}Y_{5, 2}Y_{6, 3}$ & $8  V_0 + 8  V_1 + 12  V_2$ \\ \hline
$x_{134}$ & $Y_{1, -3}Y_{1, -3}Y_{2, 2}Y_{2, 2}Y_{3, -2}Y_{4, -1}Y_{5, 2}Y_{6, 3}Y_{7, 4}$ & $4  V_0 + 5  V_1 + 6  V_2$ \\ \hline
$x_{135}$ & $Y_{1, -3}Y_{2, 2}Y_{3, -2}Y_{5, 2}Y_{7, 4}$ & $4  V_0 + 5  V_1 + 6  V_2$ \\ \hline
$x_{136}$ & $Y_{1, -3}Y_{2, 2}Y_{2, 2}Y_{3, -2}Y_{3, -2}Y_{4, -1}Y_{4, -1}Y_{5, 2}Y_{6, 3}Y_{7, 4}Y_{8, 5}$ & $8  V_0 + 8  V_1 + 12  V_2$ \\ \hline
\caption{Cluster variables with corresponding dominant monomials and their images under $\phik$ in $\Cone$ of type $E_8$.}
\label{tab:E8_cluster} 
\end{longtable}

\end{scriptsize}

\bibliographystyle{plain}
\bibliography{main}

\begin{thebibliography}{10}

\bibitem{MR1360497}
A.~Beauville.
\newblock Conformal blocks, fusion rules and the {V}erlinde formula.
\newblock In {\em Proceedings of the {H}irzebruch 65 {C}onference on {A}lgebraic {G}eometry ({R}amat {G}an, 1993)}, volume~9 of {\em Israel Math. Conf. Proc.}, pages 75--96. Bar-Ilan Univ., Ramat Gan, 1996.

\bibitem{MR4159839}
W.~Chang, B.~Duan, C.~Fraser, and J.-R. Li.
\newblock Quantum affine algebras and {G}rassmannians.
\newblock {\em Math. Z.}, 296(3-4):1539--1583, 2020.

\bibitem{MR2164838}
F.~Chapoton.
\newblock Functional identities for the {R}ogers dilogarithm associated to cluster {$Y$}-systems.
\newblock {\em Bull. London Math. Soc.}, 37(5):755--760, 2005.

\bibitem{MR1357195}
V.~Chari and A.~Pressley.
\newblock Quantum affine algebras and their representations.
\newblock In {\em Representations of groups ({B}anff, {AB}, 1994)}, volume~16 of {\em CMS Conf. Proc.}, pages 59--78. Amer. Math. Soc., Providence, RI, 1995.

\bibitem{MR1483901}
V.~Chari and A.~Pressley.
\newblock Factorization of representations of quantum affine algebras.
\newblock In {\em Modular interfaces ({R}iverside, {CA}, 1995)}, volume~4 of {\em AMS/IP Stud. Adv. Math.}, pages 33--40. Amer. Math. Soc., Providence, RI, 1997.

\bibitem{MR1300632}
Vyjayanthi Chari and Andrew Pressley.
\newblock {\em A guide to quantum groups}.
\newblock Cambridge University Press, Cambridge, 1994.

\bibitem{MR1424041}
P.~Di~Francesco, P.~Mathieu, and D.~S\'en\'echal.
\newblock {\em Conformal field theory}.
\newblock Graduate Texts in Contemporary Physics. Springer-Verlag, New York, 1997.

\bibitem{MR1887642}
S.~Fomin and A.~Zelevinsky.
\newblock Cluster algebras. {I}. {F}oundations.
\newblock {\em J. Amer. Math. Soc.}, 15(2):497--529, 2002.

\bibitem{MR2004457}
S.~Fomin and A.~Zelevinsky.
\newblock Cluster algebras. {II}. {F}inite type classification.
\newblock {\em Invent. Math.}, 154(1):63--121, 2003.

\bibitem{MR2031858}
S.~Fomin and A.~Zelevinsky.
\newblock {$Y$}-systems and generalized associahedra.
\newblock {\em Ann. of Math. (2)}, 158(3):977--1018, 2003.

\bibitem{MR1810773}
E.~Frenkel and E.~Mukhin.
\newblock Combinatorics of {$q$}-characters of finite-dimensional representations of quantum affine algebras.
\newblock {\em Comm. Math. Phys.}, 216(1):23--57, 2001.

\bibitem{FO21}
R.~Fujita and S.-j. Oh.
\newblock Q-data and representation theory of untwisted quantum affine algebras.
\newblock {\em Comm. Math. Phys.}, 384(2):1351--1407, 2021.

\bibitem{MR3282650}
A.S. Gleitz.
\newblock On the {KNS} conjecture in type {$E$}.
\newblock {\em Ann. Comb.}, 18(4):617--643, 2014.

\bibitem{MR1745263}
G.~Hatayama, A.~Kuniba, M.~Okado, T.~Takagi, and Y.~Yamada.
\newblock Remarks on fermionic formula.
\newblock In {\em Recent developments in quantum affine algebras and related topics ({R}aleigh, {NC}, 1998)}, volume 248 of {\em Contemp. Math.}, pages 243--291. Amer. Math. Soc., Providence, RI, 1999.

\bibitem{MR2254805}
D.~Hernandez.
\newblock The {K}irillov-{R}eshetikhin conjecture and solutions of {$T$}-systems.
\newblock {\em J. Reine Angew. Math.}, 596:63--87, 2006.

\bibitem{MR2682185}
D.~Hernandez and B.~Leclerc.
\newblock Cluster algebras and quantum affine algebras.
\newblock {\em Duke Math. J.}, 154(2):265--341, 2010.

\bibitem{MR3500832}
D.~Hernandez and B.~Leclerc.
\newblock A cluster algebra approach to {$q$}-characters of {K}irillov-{R}eshetikhin modules.
\newblock {\em J. Eur. Math. Soc. (JEMS)}, 18(5):1113--1159, 2016.

\bibitem{HL21}
D.~Hernandez and B.~Leclerc.
\newblock Quantum affine algebras and cluster algebras.
\newblock In {\em Interactions of quantum affine algebras with cluster algebras, current algebras and categorification---in honor of {V}yjayanthi {C}hari on the occasion of her 60th birthday}, volume 337 of {\em Progr. Math.}, pages 37--65. Birkh\"auser/Springer, Cham, [2021] \copyright 2021.

\bibitem{ICT21}
N.~Ilten, A.~Nájera~Chávez, and H.~Treffinger.
\newblock Deformation theory for finite cluster complexes, 2021.

\bibitem{MR3029994}
R.~Inoue, O.~Iyama, B.~Keller, A.~Kuniba, and T.~Nakanishi.
\newblock Periodicities of {T}-systems and {Y}-systems, dilogarithm identities, and cluster algebras {I}: type {$B_r$}.
\newblock {\em Publ. Res. Inst. Math. Sci.}, 49(1):1--42, 2013.

\bibitem{MR3029995}
R.~Inoue, O.~Iyama, B.~Keller, A.~Kuniba, and T.~Nakanishi.
\newblock Periodicities of {T}-systems and {Y}-systems, dilogarithm identities, and cluster algebras {II}: types {$C_r$}, {$F_4$}, and {$G_2$}.
\newblock {\em Publ. Res. Inst. Math. Sci.}, 49(1):43--85, 2013.

\bibitem{MR1104219}
V.~G. Kac.
\newblock {\em Infinite-dimensional {L}ie algebras}.
\newblock Cambridge University Press, Cambridge, third edition, 1990.

\bibitem{MR750341}
V.~G. Kac and D.~H. Peterson.
\newblock Infinite-dimensional {L}ie algebras, theta functions and modular forms.
\newblock {\em Adv. in Math.}, 53(2):125--264, 1984.

\bibitem{KKOP22}
M.~Kashiwara, M.~Kim, S.-j. Oh, and E.~Park.
\newblock Simply laced root systems arising from quantum affine algebras.
\newblock {\em Compos. Math.}, 158(1):168--210, 2022.

\bibitem{KKOP24}
M.~Kashiwara, M.~Kim, S.-j. Oh, and E.~Park.
\newblock Monoidal categorification and quantum affine algebras {II}.
\newblock {\em Invent. Math.}, 236(2):837--924, 2024.

\bibitem{MR2999039}
Bernhard Keller.
\newblock The periodicity conjecture for pairs of {D}ynkin diagrams.
\newblock {\em Ann. of Math. (2)}, 177(1):111--170, 2013.

\bibitem{MR947332}
A.~N. Kirillov.
\newblock Identities for the {R}ogers dilogarithmic function connected with simple {L}ie algebras.
\newblock {\em Zap. Nauchn. Sem. Leningrad. Otdel. Mat. Inst. Steklov. (LOMI)}, 164:121--133, 198, 1987.

\bibitem{MR1356515}
Anatol~N. Kirillov.
\newblock Dilogarithm identities.
\newblock {\em Progr. Theoret. Phys. Suppl.}, 118:61--142, 1995.
\newblock Quantum field theory, integrable models and beyond (Kyoto, 1994).

\bibitem{MR4327093}
S.~Kumar.
\newblock {\em Conformal blocks, generalized theta functions and the {V}erlinde formula}, volume~42 of {\em New Mathematical Monographs}.
\newblock Cambridge University Press, Cambridge, 2022.

\bibitem{MR1202213}
A.~Kuniba.
\newblock Thermodynamics of the {$U_q(X^{(1)}_r)$} {B}ethe ansatz system with {$q$} a root of unity.
\newblock {\em Nuclear Phys. B}, 389(1):209--244, 1993.

\bibitem{MR1192727}
A.~Kuniba and T.~Nakanishi.
\newblock Spectra in conformal field theories from the {R}ogers dilogarithm.
\newblock {\em Modern Phys. Lett. A}, 7(37):3487--3494, 1992.

\bibitem{MR1304818}
A.~Kuniba, T.~Nakanishi, and J.~Suzuki.
\newblock Functional relations in solvable lattice models. {I}. {F}unctional relations and representation theory.
\newblock {\em Internat. J. Modern Phys. A}, 9(30):5215--5266, 1994.

\bibitem{MR2773889}
A.~Kuniba, T.~Nakanishi, and J.~Suzuki.
\newblock {$T$}-systems and {$Y$}-systems in integrable systems.
\newblock {\em J. Phys. A}, 44(10):103001, 146, 2011.

\bibitem{MR3628223}
C.-h. Lee.
\newblock Positivity and periodicity of {$Q$}-systems in the {WZW} fusion ring.
\newblock {\em Adv. Math.}, 311:532--568, 2017.

\bibitem{MR4342381}
H.~Nakajima.
\newblock Euler numbers of {H}ilbert schemes of points on simple surface singularities and quantum dimensions of standard modules of quantum affine algebras.
\newblock {\em Kyoto J. Math.}, 61(2):377--397, 2021.

\bibitem{MR2804544}
T.~Nakanishi.
\newblock Dilogarithm identities for conformal field theories and cluster algebras: simply laced case.
\newblock {\em Nagoya Math. J.}, 202:23--43, 2011.

\bibitem{MR954762}
E.~Verlinde.
\newblock Fusion rules and modular transformations in {$2$}d conformal field theory.
\newblock {\em Nuclear Phys. B}, 300(3):360--376, 1988.

\bibitem{MR2290758}
D.~Zagier.
\newblock The dilogarithm function.
\newblock In {\em Frontiers in number theory, physics, and geometry. {II}}, pages 3--65. Springer, Berlin, 2007.

\bibitem{MR1092210}
Al.~B. Zamolodchikov.
\newblock On the thermodynamic {B}ethe ansatz equations for reflectionless {$ADE$} scattering theories.
\newblock {\em Phys. Lett. B}, 253(3-4):391--394, 1991.

\end{thebibliography}

\end{document}